\documentclass[a4paper, 11pt]{amsart}
\usepackage[T1]{fontenc}
\usepackage[utf8]{inputenc}
\usepackage[english]{babel}
\usepackage{amsmath}
\usepackage{amssymb}
\usepackage{amsfonts}
\usepackage{amsthm}
\usepackage{bm}
\usepackage{enumitem}
\usepackage{moreenum}
\usepackage{mathtools}
\usepackage{xcolor}
\usepackage{csquotes}
\usepackage{hyperref}
\usepackage{subfigure}

\numberwithin{equation}{section}
\newtheorem{theorem}{Theorem}[section]
\newtheorem{lemma}[theorem]{Lemma}

\newtheorem{proposition}[theorem]{Proposition}
\theoremstyle{definition}

\newtheorem{remark}[theorem]{Remark}

\newcommand{\per}{\mathrm{per}}
\newcommand{\e}{\mathrm{e}}
\newcommand{\R}{\mathbb R}
\newcommand{\C}{\mathbb C}

\newcommand{\N}{\mathbb N}

\newcommand{\I}{\mathcal I}
\newcommand{\F}{\mathcal F}
\newcommand{\SL}{\mathcal S}
\newcommand{\LL}{\mathcal L}
\newcommand{\T}{\mathcal T}
\newcommand{\M}{\mathcal M}

\newcommand{\Pro}{\mathcal P}
\newcommand{\A}{\mathcal A}
\newcommand{\U}{\mathcal U}
\newcommand{\D}{\mathcal D}

\newcommand{\Cc}{\mathcal C}
\DeclareMathOperator{\im}{im}
\DeclareMathOperator{\sgn}{sgn}

\newcommand{\pp}{p}

\renewcommand{\vec}[1]{\bm{#1}}

\allowdisplaybreaks

\title{Bifurcation analysis for axisymmetric capillary water waves with vorticity and swirl}

\author[Andr\'e H. Erhardt, Erik Wahlén, and Jörg Weber]{Andr\'e H. Erhardt, Erik Wahlén and Jörg Weber}

\address{Andr\'e H. Erhardt \newline
Weierstrass Institute for Applied Analysis and Stochastics, Mohrenstraße 39, 10117 Berlin, Germany}
\email{andre.erhardt@wias-berlin.de}

\address{Erik Wahlén \newline
Centre for Mathematical Sciences, Lund University, P.O.Box 118, 22100 Lund, Sweden}
\email{erik.wahlen@math.lu.se}

\address{Jörg Weber \newline
	Centre for Mathematical Sciences, Lund University, P.O.Box 118, 22100 Lund, Sweden}
\email{jorg.weber@math.lu.se}


\thanks{\today}

\begin{document}
\begin{abstract}
We study steady axisymmetric water waves with general vorticity and swirl, subject to the influence of surface tension. This can be formulated as an elliptic free boundary problem in terms of Stokes' stream function. A change of variables allows us to overcome the generic coordinate-induced singularities and to cast the problem in the form \enquote{identity plus compact}, which is amenable to Rabinowitz' global bifurcation theorem, while no restrictions regarding the absence of stagnation points in the flow have to be made. Within the scope of this new formulation, local and global solution curves, bifurcating from laminar flows with a flat surface, are constructed.
\end{abstract}
\maketitle

\section{Introduction}
In the last decades, there has been a lot of progress on the two-dimensional  steady water wave problem with vorticity (see for example \cite{C, EEW2011, Varholm20, Wahlen06, WahlenWeber21} and references therein). The corresponding three-dimensional problem is significantly more challenging, due to the lack of a general formulation which is amenable to methods from nonlinear functional analysis. This is related to the fact that in two dimensions, the vorticity is a scalar field which is constant along streamlines, while in three dimensions it is a vector field which satisfies the vorticity equation, including the vortex stretching term. One approach to at least gain some insight is to investigate flows under under certain geometrical assumptions to fill the gap between two-dimensional and three-dimensional flows. This is one of the motivations for studying the axisymmetric Euler equations, which in many ways behave like the two-dimensional equations. Indeed, for the time-dependent problem, in the swirl-free case, these possess a global existence theory for smooth solutions similar to two-dimensional flows; see \cite{Abidi, UkhovskiiIudovich68} and references therein (note however the recent remarkable result  \cite{Elgindi21} on singularity formation of non-smooth solutions). The steady axisymmetric problem is also of considerable physical importance, as it can be used to model phenomena such as jets, cavitational flows, bubbles and vortex rings (see for example \cite{AltCaffarelliFriedman83, CaffarelliFriedman82, CaoWanZhan21, DoakVandenBroeck18, FraenkelBerger74, Friedman83, Ni, Saffman92, VarvarucaWeiss14} and references therein).

In this paper, we study axisymmetric water waves with surface tension, modelled by assuming that the domain is bounded by a free surface on which capillary forces are acting, and that  in cylindrical coordinates $(r,\vartheta,z)$ the domain and flow are independent of the azimuthal variable $\vartheta$.
In the irrotational  and swirl-free setting, such waves were studied numerically  by Vanden-Broeck et al.~\cite{VB-M-S} and Osborne and Forbes \cite{OsborneForbes01}, who found similarities to two-dimensional capillary waves, including overhanging profiles and  limiting configurations with trapped bubbles at their troughs. 
The small-amplitude theory is intimately connected to Rayleigh's instability criterion for a  liquid jet \cite{Rayleigh1879} (see also \cite{HancockBush02, VB-M-S}), which says that a circular capillary jet is unstable to perturbations whose wavelength exceed the circumference of the jet. Indeed, this instability criterion is satisfied precisely when the dispersion relation for small-amplitude waves has purely imaginary solutions, while steady waves are obtained when the solutions are real \cite{VB-M-S} (that is, for smaller wavelengths). According to Hancock and Bush \cite{HancockBush02} a stationary form of such steady waves may be observed at the base of a jet which is impacting on a reservoir of the same fluid. If the reservoir is contaminated, the wave field is moved up the jet and a so called `fluid pipe' with a quiescent surface is formed at the base. We also note that in recent years there has been increased interest in waves on jets in other physical contexts, such as electrohydrodynamic flows \cite{GrandisonEtAl08} and ferrofluids \cite{BlythParau14, DoakVandenBroeck19, GrovesNilsson18}.

In this paper, we consider liquid jets with both vorticity and swirl. A motivation for this is that a viscous boundary layer in a pipe typically gives rise to vorticity, which may have a significant effect on the jet flowing out of the pipe. As an idealisation, we  assume that the jet extends indefinitely in the $z$-direction and ignore viscosity and gravity.
In the irrotational swirl-free case, the problem can be formulated in terms of a harmonic velocity potential.
In contrast, we formulate the problem in terms of Stokes' stream function, which satisfies a  second-order semilinear elliptic equation known alternatively in the literature as the \textit{Hicks equation}, the \textit{Bragg--Hawthorne equation} or the \textit{Squire--Long equation}, cf.~\cite{Saffman92}. This equation is also known from plasma physics as the \textit{Grad--Shafranov equation}, cf. \cite{ConstantinP}. The first aim of the paper is to construct small-amplitude solutions using local bifurcation theory in this more general context. In contrast to \cite{VB-M-S}, this means that the bifurcation conditions are much less explicit and that we require qualitative methods. The second aim is to construct large-amplitude solutions using global bifurcation theory and a reformulation of the problem inspired by the recent paper \cite{WahlenWeber21} on the  two-dimensional gravity-capillary water wave problem with vorticity.

We now describe the plan of the paper. First in Section~\ref{sec:description} we start by introducing the main problem we are going to study. This means we start with the incompressible Euler equations and recall its axisymmetric version. In Section~\ref{sec:preliminaries}, we discuss regularity issues and trivial solutions of the axisymmetric incompressible Euler equations. Regarding regularity issues and in order to reformulate the problem in a secure functional-analytic setting, we avoid coordinate-induced singularities by introducing a new variable in terms of the Stokes stream function and view it (partly) as a function on five-dimensional space; this trick to overcome this kind of coordinate singularities is well-known and goes back to Ni \cite{Ni}. Then, we study local bifurcations in Section~\ref{sec:localbif} in the spirit of the theorem by Crandall--Rabinowitz, mainly by introducing the so-called good unknown; the main result of this section is Theorem \ref{thm:LocalBifurcation}. In addition, in Section~\ref{sec:conditions} we take a closer look at the conditions for local bifurcation. First we establish spectral properties of the corresponding Sturm--Liouville problem of limit-point type and with boundary condition dependent on the eigenvalue. After that, we investigate some specific examples in more detail. Finally, in Section~\ref{sec:GlobalBifurcation} we close the paper by investigating global bifurcations; see Theorem \ref{thm:GlobalBifurcation}.

Since we require the radius $r$ to be a graph of the longitudinal position $z$ along the water surface, our theoretical framework, in contrast to \cite{WahlenWeber21}, does not allow for overhanging waves, and we leave it to further research to include this possibility. This would clearly be a desirable extension in view of the numerical results in \cite{OsborneForbes01, VB-M-S}.

\section{Description of the problem and the governing equations}\label{sec:description}

We consider periodic axisymmetric capillary waves travelling at constant speed along the $z$ axis. The fluid is assumed to be inviscid and incompressible. In a frame moving with the wave, the flow is therefore governed by the steady incompressible Euler equations
\begin{align}
\begin{split}
(\vec{u}\cdot\nabla)\vec{u}&=-\nabla \pp,
\\
\hfill\nabla\cdot\vec{u}&=0,
\end{split}\quad \vec{x}=(x,y,z)^T\in\Omega\subseteq\R^3\label{Euler}
\end{align}
where $\vec{u}=\vec{u}(\vec{x})$ and $\pp=\pp(\vec{x})$ denote the velocity and the pressure, respectively, and $\Omega$ is the fluid domain.
In  cylindrical coordinates $(r,\vartheta,z)$, that is, $x=r\cos\vartheta$, $y=r\sin\vartheta$ and $z=z$, the velocity field $\vec{u}$ is expressed as
$$
\vec{u}=u^r(r,z)\vec{e}_r+u^\vartheta(r,z)\vec{e}_\vartheta+u^z(r,z)\vec{e}_z,
$$
where the vectors 
$$
\vec{e}_r=\left(\frac{x}{r},\frac{y}{r},0\right)^T,\quad\vec{e}_\vartheta=\left(-\frac{y}{r},\frac{x}{r},0\right)^T\quad\text{and}\quad\vec{e}_z=\left(0,0,1\right)^T
$$
form an orthonormal basis. Note that we allow for non-zero swirl, $u^\vartheta\ne 0$. From the incompressibility and the axisymmetry of the flow it follows that we can introduce Stokes' stream function $\Psi(r,z)$, such that
$$
u^r=\frac{1}{r}\frac{\partial\Psi}{\partial z}\quad\text{and}\quad u^z=-\frac{1}{r}\frac{\partial\Psi}{\partial r}.
$$
Moreover, the quantity $ru^\vartheta$ is constant along streamlines, which we express as $ru^\vartheta=F(\Psi)$ where $F$ is an arbitrary function.  The steady Euler equations are then equivalent to the  \textit{Bragg--Hawthorne equation}
$$
-\Delta^*\Psi=r^2\gamma(\Psi)+F(\Psi)F'(\Psi),
$$
where $\gamma$ is an arbitrary function and 
\[
\Delta^*\Psi:=\Psi_{rr}-\frac1r\Psi_r+\Psi_{zz}
\]
cf.~\cite[Chapter 3.13]{Saffman92}.
Note that the corresponding vorticity vector is given by
\begin{align*}
\vec{\omega}&=
-\frac{\partial u^\vartheta}{\partial z}\vec{e}_r
+\left(\frac{\partial u^r}{\partial z}-\frac{\partial u^z}{\partial r}
\right) \vec{e}_\vartheta+\frac{1}{r} \frac{\partial (r  u^\vartheta)}{\partial r}\vec{e}_z,
\\
&=-\frac{1}{r} F'(\Psi)\Psi_z \vec{e}_r+\frac1{r} \Delta^* \Psi \vec{e}_\vartheta+\frac{1}{r} F'(\Psi)\Psi_r\vec{e}_z.
\end{align*}

We next consider the boundary conditions.  Assume that the fluid domain is given by $\Omega=\{(r,z)\in\R^2:~0<r<d+\eta(z)\}$ and its boundaries by $\partial\Omega_\mathcal{S}=\{(r,z)\in\R^2:~r=d+\eta(z)\}$ (free surface) and $\partial\Omega_\mathcal{C}=\{(r,z)\in\R^2:~r=0\}$ (center line). Although the latter could be considered as part of the domain, it is sometimes convenient to consider it as a boundary due to the appearance of inverse powers of $r$ in the equations. On the free surface $r=d+\eta(z)$ 
we have the kinematic boundary condition $\vec{u}\cdot\vec{n}=0$, where $\vec{n}=\vec{e}_r - \eta'(z) \vec{e}_z$ denotes a normal vector.
Expressed in terms of $\Psi$, this takes the form $\Psi_z+\eta_z\Psi_r=0$ on $\partial\Omega_\mathcal{S}$.
In addition, we have the dynamic boundary condition $p=-\sigma\kappa$ on $\partial\Omega_\mathcal{S}$
where
\[\kappa=\kappa[\eta]=\frac{\eta_{zz}}{(1+\eta_z^2)^{3/2}}-\frac{1}{(d+\eta)\sqrt{1+\eta_z^2}}\]
is the mean curvature of $\partial\Omega_\mathcal{S}$ and $\sigma>0$ is the coefficient of surface tension. Using Bernoulli's law we can eliminate the pressure and express this as 
$$
\frac{\Psi_r^2+\Psi_z^2+F(\Psi)^2}{2r^2}-\sigma\kappa=Q
$$
on $\partial\Omega_\mathcal{S}$, where $Q$ is the Bernoulli constant.
At the center line $\partial\Omega_\mathcal{C}$ the identity $\Psi_z=ru^r$ shows that $\Psi_z=0$.
Summarising, we have following boundary value problem:
\begin{equation}
\begin{aligned}
\Delta^*\Psi+r^2\gamma(\Psi)+F(\Psi)F'(\Psi)&=0 && \text{in } \Omega,
\\
\frac{\Psi_r^2+\Psi_z^2+F(\Psi)^2}{2r^2}-\sigma\kappa&=Q  &&\text{on }\partial\Omega_\mathcal{S},
\\
\Psi_z+\eta_z\Psi_r&=0 && \text{on } \partial\Omega_\mathcal{S},
\\
\Psi_z&=0 && \text{on } \partial\Omega_\mathcal{C},
\end{aligned}
\label{stationary_axisymmetric_problem}
\end{equation}
%
where  $F$ and $\gamma$ are arbitrary functions of $\Psi$.

\section{Preliminaries}\label{sec:preliminaries}
\subsection{The equations} The last two boundary conditions in \eqref{stationary_axisymmetric_problem} mean that $\Psi$ is constant on both $\partial\Omega_\mathcal{S}$ and $\partial\Omega_\mathcal{C}$. We normalise $\Psi$ such that it vanishes on $\partial\Omega_\mathcal{C}$ and assign the name $m$ to its value on $\partial\Omega_\mathcal{S}$. Thus, we shall deal with the equations
\begin{subequations}
\begin{align}
	\Psi_{rr}-\frac1r\Psi_r+\Psi_{zz}&=-r^2\gamma(\Psi)-F(\Psi)F'(\Psi) &&\text{in }\Omega,\\
	\frac{\Psi_r^2+\Psi_z^2+F(\Psi)^2}{2r^2}-\sigma\kappa&=Q && \text{on }\partial\Omega_\mathcal{S},\\
	\Psi&=m &&\text{on }\partial\Omega_\mathcal{S},\\
	\Psi&=0 &&\text{on }\partial\Omega_\mathcal{C},\label{eq:Psi_r=0}
\end{align}
\end{subequations}
where $Q$ and $m$ are constants. The fluid velocity is given by
\begin{align}\label{eq:relation_velocity_Psi}
	\vec{u}=\frac{F(\Psi)}{r}\vec{e}_\vartheta-\nabla\times(\Psi\vec{e}_\vartheta/r)=\frac{\Psi_z}{r}\vec{e}_r+\frac{F(\Psi)}{r}\vec{e}_\vartheta-\frac{\Psi_r}{r}\vec{e}_z.
\end{align}
Following a trick of Ni \cite{Ni}, we introduce the function $\psi$ via
\begin{align}\label{eq:relation_Psi_psi}
	\Psi=r^2\psi,
\end{align}
In terms of $\psi$, the equations read
\begin{subequations}\label{eq:OriginalEquation_psi}
\begin{align}
	\psi_{rr}+\frac3r\psi_r+\psi_{zz}&=-\gamma(r^2\psi)-\frac{1}{r^2}F(r^2\psi)F'(r^2\psi)&&\text{in }\Omega,\label{eq:OriginalEquation_psi_PDE}\\
	\frac{r^2(\psi_r^2+\psi_z^2)}{2}+\frac{F(r^2\psi)^2}{2r^2}+\frac{2m\psi_r}{r}+\frac{2m^2}{r^4}-\sigma\kappa&=Q &&\text{on }\partial\Omega_\mathcal{S},\label{eq:OriginalEquation_psi_Bernoulli}\\
	\psi&=\frac{m}{r^2} &&\text{on }\partial\Omega_\mathcal{S}.\label{eq:OriginalEquation_psi_top}
\end{align}
\end{subequations}
Notice that we no longer need to impose a condition on $r=0$, provided $\psi$ is continuous at $r=0$, since then \eqref{eq:Psi_r=0} is automatically satisfied for $\Psi$ given by \eqref{eq:relation_Psi_psi}.

\subsection{Regularity issues}
Quite naturally, the fluid velocity $\vec{u}$ should be at least of class $C^1$ (in Cartesian coordinates). Written in terms of $\psi$, \eqref{eq:relation_velocity_Psi} reads
\begin{align}\label{eq:relation_velocity_psi}
	\vec{u}=\frac{F(r^2\psi)}{r}\vec{e}_\vartheta-\nabla\times(r\psi\vec{e}_\vartheta)=r\psi_z\vec{e}_r+\frac{F(r^2\psi)}{r}\vec{e}_\vartheta-(2\psi+r\psi_r)\vec{e}_z.
\end{align}
Due to \cite{Liu2009}, $\vec{u}$ is of class $C^1$ provided $F(r^2\psi)/r$ is of class $C^1$ and $r\psi$ is of class $C^2$, both viewed as functions on $\{(r,z)\in[0,\infty)\times\R:r\le d+\eta(z)\}$, and, moreover, $F(r^2\psi)/r$, $r\psi$, and $(r\psi)_{rr}$ vanish at $r=0$. In view of
\[(r\psi)_r=\psi+r\psi_r,\quad(r\psi)_{rr}=2\psi_r+r\psi_{rr},\]
and
\[(F(r^2\psi)/r)_r=F'(r^2\psi)(2\psi+r\psi_r)-F(r^2\psi)/r^2,\]
it is therefore sufficient to assume
\begin{align}\label{eq:regularity_condition_psi}
	\psi\in C^2(\overline\Omega),\quad \psi_r\big|_{r=0}=0,\
\end{align}
and
\[F\in C^1(\R),\quad F(0)=0.\]
Furthermore, we shall need that the right-hand side of \eqref{eq:OriginalEquation_psi_PDE} is in a Hölder class $C^{0,\alpha}$ if $\psi$ is $C^{0,\alpha}$. To this end, it is sufficient that both $F'$ and $G(x)\coloneqq F(x)/x$ (continuously extended to $x=0$ by $G(0)\coloneqq F'(0)$) are locally Lipschitz continuous in view of
\begin{align}\label{eq:relFG}
	\frac{1}{r^2}F(r^2\psi)=G(r^2\psi)\psi.
\end{align}
Moreover, the nonlinear operator $\F$ introduced later should be of class $C^2$. Hence, we need that $\gamma$, $F$, and $F'$ are locally of class $C^{2,1}$; notice that this condition on $F$ already implies the desired property of $G$ as above. Also, in order to construct trivial solutions, we will need a Lipschitz property of $\gamma$ and $FF'$. Overall we impose the following assumptions on $\gamma$ and $F$:

\begin{align}\label{eq:assumptions_gamma_F}
	\gamma\in C_{\text{loc}}^{2,1}(\R),\quad F\in C_{\text{loc}}^{3,1}(\R),\quad\|\gamma'\|_\infty<\infty,\quad\|(FF')'\|_\infty<\infty,\quad F(0)=0.
\end{align}

\subsection{Trivial solutions}\label{sec:trivial}
We now have a look at trivial solutions of \eqref{eq:OriginalEquation_psi}, that is, solutions of \eqref{eq:OriginalEquation_psi} independent of $z$. Therefore, we consider the (singular) Cauchy problem
\begin{subequations}\label{eq:trivial}
\begin{align}
	\psi_{rr}+\frac3r\psi_r&=-\gamma(r^2\psi)-\frac{1}{r^2}F(r^2\psi)F'(r^2\psi)\quad\text{on }(0,d],\label{eq:trivial_ode}\\
	\psi(0)&=\lambda,\label{eq:trivial_initial_value}\\
	\psi_r(0)&=0.\label{eq:trivial_initial_derivative}
\end{align}
\end{subequations}
Here, $\lambda\in\R$ is a parameter, which will later serve as the bifurcation parameter, and \eqref{eq:trivial_initial_derivative} is imposed due to \eqref{eq:regularity_condition_psi}. Notice that, in view of \eqref{eq:relation_velocity_psi}, there is a one-to-one correspondence of the parameter $\lambda$ and the velocity at the symmetry axis via $\vec{u}=-2\lambda\vec{e}_z$ at $r=0$.

In order to solve \eqref{eq:trivial}, we rewrite \eqref{eq:trivial}, making use of \eqref{eq:trivial_ode}, \eqref{eq:trivial_initial_value}, and $\partial_{rr}+\frac{3}{r}\partial_r=r^{-3}\partial_r(r^3\partial_r)$, as the integral equation
\begin{align}\label{eq:trivial_integral_equation}
	\psi(r)=\lambda-\int_0^rt^{-3}\int_0^t\left(s^3\gamma(s^2\psi(s))+s(FF')(s^2\psi(s))\right)\,ds\,dt.
\end{align}
By Lipschitz continuity of $\gamma$ and $FF'$, it is straightforward to see that the right-hand side of \eqref{eq:trivial_integral_equation} gives rise to a contraction on $C([0,\varepsilon])$ if $\varepsilon>0$ is small enough. Thus, \eqref{eq:trivial_integral_equation} has a unique continuous solution on such $[0,\varepsilon]$ by Banach's fixed point theorem. It is clear that, by virtue of \eqref{eq:relFG} and \eqref{eq:trivial_integral_equation}, this solution is of class $C^2$ on $[0,\varepsilon]$, and satisfies \eqref{eq:trivial_ode} on $(0,\varepsilon]$ and \eqref{eq:trivial_initial_value}, \eqref{eq:trivial_initial_derivative}. Now, once having left the singular point $r=0$, it is obvious that $\psi$ can be uniquely extended to a $C^2$-solution of \eqref{eq:trivial_ode} on $(0,d]$, since $\gamma$ and $FF'$ are Lipschitz continuous. Moreover, $\psi\in C^{2,1}([0,d])$ in view of \eqref{eq:trivial_ode}, \eqref{eq:trivial_initial_derivative}. 

Finally, motivated by the flattening considered below, we define
\begin{align}\label{eq:psi^lambda}
	\psi^\lambda(s)\coloneqq\psi(sd),\quad s\in[0,1]
\end{align}
where $\psi$ is the unique solution of \eqref{eq:trivial} as obtained above.

\subsection{Working in 5D and flattening}
In the following, for a function $\psi=\psi(r,z)$ on some $\Omega\subset\R^2$ we denote by $\I\psi$ the function given by
\[\I\psi(x,z)=\psi(|x|,z)\]
and defined on the set $\Omega^\I$, which results from rotating $\Omega$ around the $z$-axis in $\R^5=\{(x,z)\in\R^4\times\R\}$. Conversely, any axially symmetric set in $\R^5$ can be written as $\Omega^\I$ for a suitable $\Omega\subset\R^2$, and any axially symmetric function $\tilde\psi$ on $\Omega^\I$ equals $\I\psi$ for a certain function $\psi$ on $\Omega$, i.e, $\psi=\I^{-1}\tilde\psi$, where $\I^{-1}$ is defined on the set of axially symmetric functions. Thus, it is easy to see that $\psi$ satisfies \eqref{eq:regularity_condition_psi} and solves \eqref{eq:OriginalEquation_psi_PDE}, \eqref{eq:OriginalEquation_psi_top} if and only if $\I\psi\in C^2(\overline{\Omega^\I})$ and solves
\begin{subequations}\label{eq:PDE_in_5D}
\begin{align}
	\Delta_5\I\psi&=-\gamma(|x|^2\I\psi)-\frac{1}{|x|^2}F(|x|^2\I\psi)F'(|x|^2\I\psi)&\text{in }\Omega^\I,\\
	\I\psi&=\frac{m}{|x|^2}&\text{on }\partial\Omega_\mathcal{S}^\I,
\end{align}	
\end{subequations}
with $\Delta_5$ denoting the Laplacian in five dimensions. Therefore, no longer a term which is singular on the symmetry axis appears (cf. \eqref{eq:relFG}) -- this is the main motivation for working with $\psi$ instead of $\Psi$. In order to transform \eqref{eq:PDE_in_5D} into a fixed domain, we consider the flattening $(x,z)\mapsto(y,z)=(x/(d+\eta(z)),z)$ -- from now on, we shall always assume that $\eta>-d$. Thus, introducing $\tilde\psi$ via $\tilde\psi(y,z)=\I\psi(x,z)$, \eqref{eq:PDE_in_5D} is transformed into
\begin{subequations}\label{eq:PDE+Dirichlet_flattened_tildepsi}
\begin{align}
	L^\eta\tilde\psi&=-\gamma((d+\eta)^2|y|^2\tilde\psi)-\frac{1}{(d+\eta)^2|y|^2}(FF')((d+\eta)^2|y|^2\tilde\psi)&\text{in }\Omega_0^\I,\\
	\tilde\psi&=\frac{m}{(d+\eta)^2}&\text{on }|y|=1,
\end{align}
\end{subequations}
where $\Omega_0\coloneqq [0,1)\times\R$ and
\begin{align*}
	L^\eta\tilde\psi\coloneqq\tilde\psi_{zz}+\frac{1}{(d+\eta)^2}\left(\tilde\psi_{y_iy_i}-2(d+\eta)\eta_zy_i\tilde\psi_{y_iz}+\eta_z^2y_iy_j\tilde\psi_{y_iy_j}-((d+\eta)\eta_{zz}-2\eta_z^2)y_i\tilde\psi_{y_i}\right);
\end{align*}
here and throughout this paper, repeated indices are summed over. It is straightforward to see that $L^\eta$ is a uniformly elliptic operator, provided $d+\eta$ is uniformly bounded from below by a positive constant.

As for Bernoulli's equation \eqref{eq:OriginalEquation_psi_Bernoulli}, we do not have to take a detour and increase the dimension, since in \eqref{eq:OriginalEquation_psi_Bernoulli} no singular term appears, at least whenever the surface does not intersect the symmetry axis. Therefore, here we consider the flattening \[H[\eta]^{-1}\colon\Omega\to\Omega_0,\quad(r,z)\mapsto(s,z)=(r/(d+\eta(z)),z);\]
we shall call $H[\eta]$ the inverse map. Then, with $\bar\psi(s,z)=\psi(r,z)$, that is, 
\begin{align}\label{eq:barpsi_tildepsi}
	\I\bar\psi=\tilde\psi,
\end{align}
\eqref{eq:OriginalEquation_psi_Bernoulli} is transformed into
\begin{align}\label{eq:Bernoulli_flattened_barpsi}
	\frac{\bar\psi_s^2+((d+\eta)\bar\psi_z-\eta_z\bar\psi_s)^2}{2}+\frac{F((d+\eta)^2\bar\psi)^2}{2(d+\eta)^2}+\frac{2m\bar\psi_s}{(d+\eta)^2}+\frac{2m^2}{(d+\eta)^4}-\sigma\kappa[\eta]=Q\quad\text{on }s=1.
\end{align}

\subsection{Reformulation}
For later reasons, it is convenient to work with functions $\phi$ satisfying $\phi=0$ on $s=1$ instead of functions $\bar\psi$ with variable boundary condition at $s=1$. Thus, we introduce, for any $\lambda\in\R$, the function \[\phi=\bar\psi-\frac{d^2}{(d+\eta)^2}\psi^\lambda.\]
In terms of $\phi$, \eqref{eq:PDE+Dirichlet_flattened_tildepsi} and \eqref{eq:Bernoulli_flattened_barpsi}, furnished with \eqref{eq:barpsi_tildepsi} and $m=m(\lambda)\coloneqq d^2\psi^\lambda(1)$, read
\begin{subequations}\label{eq:PDE+Dirichlet_flattened}
	\begin{align}
	L^\eta\I\phi&=-\gamma\left((d+\eta)^2|y|^2\left(\I\phi+\frac{d^2}{(d+\eta)^2}\I\psi^\lambda\right)\right)\nonumber\\
	&\phantom{=\;}-\frac{1}{(d+\eta)^2|y|^2}(FF')\left((d+\eta)^2|y|^2\left(\I\phi+\frac{d^2}{(d+\eta)^2}\I\psi^\lambda\right)\right)-L^\eta\frac{d^2\I\psi^\lambda}{(d+\eta)^2}&\text{in }\Omega_0^\I,\\
	\I\phi&=0&\text{on }|y|=1,
	\end{align}
\end{subequations}
and
\begin{align}\label{eq:Bernoulli_flattened}
&\frac{\left(\phi_s+\frac{d^2}{(d+\eta)^2}\psi^\lambda_s\right)^2+\left((d+\eta)\phi_z-\eta_z\left(\phi_s+\frac{2m(\lambda)+d^2\psi^\lambda_s}{(d+\eta)^2}\right)\right)^2}{2}\nonumber\\
&+\frac{F\left((d+\eta)^2\left(\phi+\frac{m(\lambda)}{(d+\eta)^2}\right)\right)^2}{2(d+\eta)^2}+\frac{2m(\lambda)\left(\phi_s+\frac{d^2}{(d+\eta)^2}\psi^\lambda_s\right)}{(d+\eta)^2}+\frac{2m(\lambda)^2}{(d+\eta)^4}-\sigma\kappa[\eta]=Q\quad\text{ on }s=1.
\end{align}
Henceforth, we search for solutions $(\lambda,\eta,\phi)$ of \eqref{eq:PDE+Dirichlet_flattened}, \eqref{eq:Bernoulli_flattened}.

Our goal is to rewrite \eqref{eq:PDE+Dirichlet_flattened}, \eqref{eq:Bernoulli_flattened} in the form \enquote{identity plus compact}, namely, as $(\eta,\phi)=\M(\lambda,\eta,\phi)$ with $\M$ compact. Meanwhile, we shall also clarify what $Q$ exactly is, namely, we define it as an expression in $(\lambda,\eta,\phi)$. To this end, we first fix $0<\alpha<1$ and introduce the Banach space
\[X\coloneqq\left\{(\eta,\phi)\in C_{0,\per,\e}^{2,\alpha}(\R)\times C_{\per,\e}^{0,\alpha}(\overline{\Omega_0}):\phi=0\text{ on }s=1,\ \I\phi\in H^1_{\per,\e}(\Omega_0^\I)\right\},\]
equipped with the canonical norm
\[\|(\eta,\phi)\|_X=\|\eta\|_{C^{2,\alpha}_\per(\R)}+\|\phi\|_{C^{0,\alpha}_\per(\Omega_0)}+\|\I\phi\|_{H^1_\per(\Omega_0^\I)}.\]
Here, the indices \enquote{$\per$}, \enquote{$\e$}, and \enquote{$0$} denote $L$-periodicity ($\nu\coloneqq 2\pi/L$ in the following), evenness (in $z$ with respect to $z=0$), and zero average over one period.

First, for
\[(\lambda,\eta,\phi)\in\R\times\U\coloneqq\{(\lambda,\eta,\phi)\in\R\times X:d+\eta>0\text{ on }[0,L]\},\]
we let $\A(\lambda,\eta,\phi)=\I^{-1}\varphi$, where $\varphi\in C^{2,\alpha}(\overline{\Omega_0^\I})$ is the unique solution of
\begin{subequations}\label{eq:A_system}
	\begin{align}
	L^\eta\varphi&=-L^\eta\frac{d^2\I\psi^\lambda}{(d+\eta)^2}-\gamma\left((d+\eta)^2|y|^2\left(\I\phi+\frac{d^2}{(d+\eta)^2}\I\psi^\lambda\right)\right)\nonumber\\
	&\phantom{=\;}-\frac{1}{(d+\eta)^2|y|^2}(FF')\left((d+\eta)^2|y|^2\left(\I\phi+\frac{d^2}{(d+\eta)^2}\I\psi^\lambda\right)\right)&\text{in }\Omega_0^\I,\label{eq:A_PDE}\\
	\varphi&=0&\text{on }|y|=1;
	\end{align}
\end{subequations}
here, notice that the right-hand side of \eqref{eq:A_PDE} is an element of $C^{0,\alpha}(\overline{\Omega_0^\I})$ (cf. \eqref{eq:relFG} and the discussion there) and that \eqref{eq:A_system} is invariant under rotations about the $z$-axis, so that $\varphi$ has to be axially symmetric.

Second, we rewrite \eqref{eq:Bernoulli_flattened} as an equation for $\eta_{zz}$, using $\A=\A(\lambda,\eta,\phi)$ instead of $\phi$ -- notice that this change does not affect the equivalence of the whole reformulation to the original equations since clearly \eqref{eq:PDE+Dirichlet_flattened} is equivalent to $\phi=\A(\lambda,\eta,\phi)$:
\begin{align*}
	\eta_{zz}&=\sigma^{-1}(1+\eta_z^2)^{3/2}\Bigg(\frac{\sigma}{(d+\eta)\sqrt{1+\eta_z^2}}+\frac{\left(\A_s+\frac{d^2}{(d+\eta)^2}\psi^\lambda_s\right)^2+\left((d+\eta)\A_z-\eta_z\left(\A_s+\frac{2m(\lambda)+d^2\psi^\lambda_s}{(d+\eta)^2}\right)\right)^2}{2}\\
	&\omit\hfill$\displaystyle+\frac{F\left((d+\eta)^2\left(\A+\frac{m(\lambda)}{(d+\eta)^2}\right)\right)^2}{2(d+\eta)^2}+\frac{2m(\lambda)\left(\A_s+\frac{d^2}{(d+\eta)^2}\psi^\lambda_s\right)}{(d+\eta)^2}+\frac{2m(\lambda)^2}{(d+\eta)^4}-Q\Bigg)$
\end{align*}
on $s=1$. In order to apply $\partial_z^{-2}\colon C_{0,\per}^{0,\alpha}(\R)\to C_{0,\per}^{2,\alpha}(\R)$, the inverse operation to twice differentiation, to this relation, the right-hand side needs to have zero average over one period. Therefore, we view $Q$ as a function of $(\lambda,\eta,\phi)$ via
\begin{align*}
	Q(\lambda,\eta,\phi)&\coloneqq\frac{1}{\langle(1+\eta_z^2)^{3/2}\rangle}\Bigg\langle(1+\eta_z^2)^{3/2}\Bigg(\frac{\sigma}{(d+\eta)\sqrt{1+\eta_z^2}}\\
	&\omit\hfill$\displaystyle+\frac{\left(\SL\A_s+\frac{d^2}{(d+\eta)^2}\SL\psi^\lambda_s\right)^2+\left((d+\eta)\SL\A_z-\eta_z\left(\SL\A_s+\frac{2m(\lambda)+d^2\SL\psi^\lambda_s}{(d+\eta)^2}\right)\right)^2}{2}$\\
	&\omit\hfill$\phantom{\coloneqq\;}\displaystyle+\frac{F\left((d+\eta)^2\left(\SL\A+\frac{m(\lambda)}{(d+\eta)^2}\right)\right)^2}{2(d+\eta)^2}+\frac{2m(\lambda)\left(\SL\A_s+\frac{d^2}{(d+\eta)^2}\SL\psi^\lambda_s\right)}{(d+\eta)^2}+\frac{2m(\lambda)^2}{(d+\eta)^4}\Bigg)\Bigg\rangle.$
\end{align*}
Here and in the following, $\langle f\rangle$ denotes the average of a $L$-periodic function $f$ over one period, and $\SL f$ denotes the evaluation of a function $f$ at $s=1$.

Putting everything together, we reformulate \eqref{eq:PDE+Dirichlet_flattened}, \eqref{eq:Bernoulli_flattened} as 
\begin{align}\label{eq:F=0}
\F(\lambda,\eta,\phi)=0
\end{align}
for $(\lambda,\eta,\phi)\in\R\times\U$, where
\[\F\colon\R\times\U\to X,\;\F(\lambda,\eta,\phi)=(\eta,\phi)-\M(\lambda,\eta,\phi),\]
with $\M=(\M^1,\M^2)$,
\begin{align*}
	&\M^1(\lambda,\eta,\phi)\coloneqq\\
	&\partial_z^{-2}\Bigg(\sigma^{-1}(1+\eta_z^2)^{3/2}\Bigg(\frac{\sigma}{(d+\eta)\sqrt{1+\eta_z^2}}\\
	&\omit\hfill$\displaystyle+\frac{\left(\SL\A_s+\frac{d^2}{(d+\eta)^2}\SL\psi^\lambda_s\right)^2+\left((d+\eta)\SL\A_z-\eta_z\left(\SL\A_s+\frac{2m(\lambda)+d^2\SL\psi^\lambda_s}{(d+\eta)^2}\right)\right)^2}{2}$\\
	&\phantom{\partial_z^{-2}\Bigg(}+\frac{F\left((d+\eta)^2\left(\SL\A+\frac{m(\lambda)}{(d+\eta)^2}\right)\right)^2}{2(d+\eta)^2}+\frac{2m(\lambda)\left(\SL\A_s+\frac{d^2}{(d+\eta)^2}\SL\psi^\lambda_s\right)}{(d+\eta)^2}+\frac{2m(\lambda)^2}{(d+\eta)^4}-Q(\lambda,\eta,\phi)\Bigg)\Bigg),\\
	&\M^2(\lambda,\eta,\phi)\coloneqq\A(\lambda,\eta,\phi).
\end{align*}

Notice that $\F$ is well-defined; in particular, both $\M^1(\lambda,\eta,\phi)$ and $\M^2(\lambda,\eta,\phi)$ are periodic and even in $z$. We summarise our reformulation in the following lemma.
\begin{lemma}
	A tuple $(\lambda,\eta,\phi)\in\R\times X$ satisfying $\eta>-d$ solves \eqref{eq:F=0} if and only if
	\begin{enumerate}[label=(\roman*)]
		\item $\eta$ and $\phi$ are of class $C^{2,\alpha}$; and
		\item the tuple
		\[(\eta,\psi)=\left(\eta,\left(\phi+\frac{d^2}{(d+\eta)^2}\psi^\lambda\right)\circ H[\eta]^{-1}\right)\]
		solves \eqref{eq:OriginalEquation_psi} with $\Omega=H[\eta](\Omega_0)$, $Q=Q(\lambda,\eta,\phi)$, and $m=m(\lambda)$; and
		\item $\psi\in C^{2,\alpha}_{\per,\e}(\overline\Omega)$ and satisfies \eqref{eq:regularity_condition_psi}.
	\end{enumerate}
\end{lemma}
\begin{proof}
	We only need to take care of the regularity properties and \eqref{eq:regularity_condition_psi}. However, to this end it is sufficient to notice that $\eta\in C^{2,\alpha}(\R)$ and $\I\phi\in C^{2,\alpha}(\overline{\Omega_0^\I})$ provided $\F(\lambda,\eta,\phi)=0$. Indeed, we have $\I\psi\in C^{2,\alpha}(\overline{\Omega^\I})$ in this case since $\I\psi^\lambda\in C^{2,1}(\overline{\Omega_0^\I})$; in particular, $\psi$ satisfies \eqref{eq:regularity_condition_psi}.
\end{proof}
By construction, all points of the form $(\lambda,0,0)$ are solutions of \eqref{eq:F=0} -- they make up the curve of trivial solutions. An inspection of $\M$ shows the following; in particular, \eqref{eq:F=0} has the form \enquote{identity plus compact}.
\begin{lemma}\label{lma:M_prop}
	$\M$ and thus $\F$ is of class $C^2$ on $\R\times\U$. Moreover, $\M$ is compact on
	\[\R\times\U_\varepsilon\coloneqq\{(\lambda,\eta,\phi)\in\R\times X:d+\eta\ge\varepsilon\text{ on }\R\}\]
	for each $\varepsilon>0$.
\end{lemma}
\begin{proof}
	The other operations in the definition of $\M$ being smooth, the property that $\M$ is of class $C^2$ follows from the property that $\A$ is of class $C^2$; this, in turn, is guaranteed by the assumption \eqref{eq:assumptions_gamma_F}. Now let $(\lambda,\eta,\phi)\in\R\times\U_\varepsilon$ be arbitrary. In the following, the quantities $C$ can change from line to line, but are always shorthand for a certain expression in its arguments which remains bounded for bounded arguments. Moreover, let $R>0$ and suppose $\|(\lambda,\eta,\phi)\|_{\R\times X}\le R$. Since $\psi^\lambda$ is of class $C^1$ with respect to $\lambda$ and $L^\eta$ is elliptic uniformly in $\eta$ due to $\eta+d\ge\varepsilon$, we see that
	\begin{align*}
		\|\I\A(\lambda,\eta,\phi)\|_{C^{2,\alpha}_\per(\overline{\Omega_0^\I})}\le C\left(R,\varepsilon^{-1},\|\gamma'\|_{L^\infty([-C(R),C(R)])},\|GF'\|_{C^{0,1}([-C(R),C(R)])}\right)
	\end{align*}
	by applying a standard Schauder estimate. This shows that $\M^2$ is compact on $\R\times\U_\varepsilon$ because of the compact embedding of $C^{2,\alpha}_\per(\overline{\Omega_0^\I})$ in $H^1_\per(\Omega_0^\I)$ and in $C^{0,\alpha}_\per(\overline{\Omega_0^\I})$ combined with \[\|f\|_{C^{0,\alpha}_\per(\overline{\Omega_0})}\le\|\I f\|_{C^{0,\alpha}_\per(\overline{\Omega_0^\I})},\quad f\in C^{0,\alpha}_\per(\overline{\Omega_0}).\]
	As for $\M^1$, we immediately find, in view of the obtained estimates for $\A$,
	\begin{align*}
		&\|\M^1(\lambda,w,\phi)\|_{C^{3,\alpha}([0,L])}\\
		&\le C\left(R,\varepsilon^{-1},\|\gamma'\|_{L^\infty([-C(R),C(R)])},\|GF'\|_{C^{0,1}([-C(R),C(R)])},\|FF'\|_{L^\infty([-C(R),C(R)])}\right).
	\end{align*}
	Hence, also $\M_1$ is compact on $\R\times\U_\varepsilon$ since $C^{3,\alpha}([0,L])$ is compactly embedded in $C^{2,\alpha}([0,L])$.
\end{proof}

\section{Local bifurcation}\label{sec:localbif}
\subsection{Computing derivatives}
We now want to calculate the partial derivative $\F_{(\eta,\phi)}$ and, in particular, its evaluation at a trivial solution. For simplicity, we shall always write $\A_\eta$ for $\A_\eta(\lambda,\eta,\phi)\delta\eta$, that is, the partial derivative of $\A$ with respect to $\eta$ evaluated at $(\lambda,\eta,\phi)$ and applied to a direction $\delta\eta$. The same applies similarly to expressions like $\A_\phi$, $L_\eta^\eta$ etc.

Linearizing the operator $L^\eta$, which only depends on $\eta$ and not on $\phi$, leads to
\begin{align*}
	L^\eta_\eta\varphi&=-\frac{2\delta\eta}{(d+\eta)^3}\left(\varphi_{y_iy_i}-2(d+\eta)\eta_zy_i\varphi_{y_iz}+\eta_z^2y_iy_j\varphi_{y_iy_j}-((d+\eta)\eta_{zz}-2\eta_z^2)y_i\varphi_{y_i}\right)\\
	&\phantom{=\;}+\frac{1}{(d+\eta)^2}\Big(-2(\eta_z\delta\eta+(d+\eta)\delta\eta_z)y_i\varphi_{y_iz}+2\eta_z\delta\eta_zy_iy_j\varphi_{y_iy_j}\\
	&\omit\hfill$\displaystyle-(\eta_{zz}\delta\eta+(d+\eta)\delta\eta_{zz}-4\eta_z\delta\eta_z)y_i\varphi_{y_i}\Big)$.
\end{align*}
Since formally linearizing an equation like $L\varphi=f$ gives $L\delta\varphi+\delta L\varphi=\delta f$, we see that $\I\A_\eta$ is the unique solution of
\begin{align*}
	L^\eta\I\A_\eta&=-L^\eta_\eta\I\A-L^\eta_\eta\frac{d^2\I\psi^\lambda}{(d+\eta)^2}+2L^\eta\frac{d^2\I\psi^\lambda\delta\eta}{(d+\eta)^3}\\
	&\phantom{=\;}-2\gamma'\left((d+\eta)^2|y|^2\left(\I\phi+\frac{d^2}{(d+\eta)^2}\I\psi^\lambda\right)\right)(d+\eta)\delta\eta|y|^2\I\phi\\
	&\phantom{=\;}+\frac{2\delta\eta}{(d+\eta)^3|y|^2}(FF')\left((d+\eta)^2|y|^2\left(\I\phi+\frac{d^2}{(d+\eta)^2}\I\psi^\lambda\right)\right)\\
	&\phantom{=\;}-2(FF')'\left((d+\eta)^2|y|^2\left(\I\phi+\frac{d^2}{(d+\eta)^2}\I\psi^\lambda\right)\right)\frac{\delta\eta\I\phi}{d+\eta}&\text{in }\Omega_0^\I,\\
	\I\A_\eta&=0&\text{on }|y|=1.
\end{align*}
Similarly, $\I\A_\phi$ is the unique solution of
\begin{align*}
	L^\eta\I\A_\phi&=-\gamma'\left((d+\eta)^2|y|^2\left(\I\phi+\frac{d^2}{(d+\eta)^2}\I\psi^\lambda\right)\right)(d+\eta)^2|y|^2\I\delta\phi\\
	&\phantom{=\;}-(FF')'\left((d+\eta)^2|y|^2\left(\I\phi+\frac{d^2}{(d+\eta)^2}\I\psi^\lambda\right)\right)\I\delta\phi&\text{in }\Omega_0^\I,\\
	\I\A_\phi&=0&\text{on }|y|=1.
\end{align*}
Evaluated at a trivial solution $(\lambda,0,0)$, we can simplify as follows:
\begin{align*}
	L^\eta\varphi&=\varphi_{zz}+\frac{1}{d^2}\varphi_{y_iy_i},\\
	L^\eta_\eta\varphi&=-\frac{2\delta\eta}{d^3}\varphi_{y_iy_i}-\frac{2\delta\eta_z}{d}y_i\varphi_{y_iz}-\frac{\delta\eta_{zz}}{d}y_i\varphi_{y_i}.
\end{align*}
In the following, we denote
\[\Delta_d\coloneqq\partial_{zz}+\frac{\partial_{y_iy_i}}{d^2}.\]
Moreover, since $\A=0$ here, we have
\begin{subequations}\label{eq:Aeta_tr}
\begin{align}
	\Delta_d\I\A_\eta&=-L^\eta_\eta\I\psi^\lambda+\frac2d\Delta_d(\I\psi^\lambda\delta\eta)+\frac{2}{d^3|y|^2}(FF')(d^2|y|^2\I\psi^\lambda)\delta\eta\nonumber\\
	&=\frac{4(\I\psi^\lambda)_{y_iy_i}}{d^3}\delta\eta+\frac{2\I\psi^\lambda+y_i(\I\psi^\lambda)_{y_i}}{d}\delta\eta_{zz}+\frac{2}{d^3|y|^2}(FF')(d^2|y|^2\I\psi^\lambda)\delta\eta&\text{in }\Omega_0^\I,\\
\I\A_\eta&=0&\text{on }|y|=1,
\end{align}
\end{subequations}
and
\begin{subequations}\label{eq:Aphi_tr}
\begin{align}
	\Delta_d\I\A_\phi&=-\left(d^2|y|^2\gamma'(d^2|y|^2\I\psi^\lambda)+(FF')'(d^2|y|^2\I\psi^\lambda)\right)\I\delta\phi&\text{in }\Omega_0^\I,\\
	\I\A_\phi&=0&\text{on }|y|=1.
\end{align}
\end{subequations}

Next, we turn to $\M^1$. After a lengthy computation we get the following results for the partial derivatives of $\M^1$ evaluated at a trivial solution $(\lambda,0,0)$, noticing that $\SL\A_\eta=\SL\A_\phi=0$ at such points:
\begin{align*}
	\M^1_\eta&=-\sigma^{-1}\left(\frac{\sigma}{d^2}+\frac{2}{d}(\SL\psi^\lambda_s)^2+\frac{F(m(\lambda))^2}{d^3}+\frac{8m(\lambda)\SL\psi^\lambda_s}{d^3}+\frac{8m(\lambda)^2}{d^5}\right)\partial_z^{-2}\delta\eta\\
	&\phantom{=\;}+\sigma^{-1}\left(\SL\psi^\lambda_s+\frac{2m(\lambda)}{d^2}\right)\partial_z^{-2}\Pro\SL\A_{\eta s},\\
	\M^1_\phi&=\sigma^{-1}\left(\SL\psi^\lambda_s+\frac{2m(\lambda)}{d^2}\right)\partial_z^{-2}\Pro\SL\A_{\phi s},
\end{align*}
where $\Pro$ is the projection onto the space of functions with zero average.

It will be convenient to introduce the abbreviation
\[c(\lambda)\coloneqq\SL(2\psi^\lambda+s\psi^\lambda_s)=\SL(2\psi^\lambda+\psi^\lambda_s)=\frac{2m(\lambda)}{d^2}+\SL\psi^\lambda_s.\]
Notice that $-c(\lambda)$ is the $z$-component of the velocity at the surface of the trivial laminar flow corresponding to $\lambda$ in view of \eqref{eq:relation_velocity_psi}. With this, we can rewrite
\begin{align}
	\M^1_\eta&=-\frac{1}{\sigma d^3}\left(\sigma d+2d^2c(\lambda)^2+F(m(\lambda))^2\right)\partial_z^{-2}\delta\eta+\sigma^{-1}c(\lambda)\partial_z^{-2}\Pro\SL\A_{\eta s},\label{eq:M1eta_tr}\\
	\M^1_\phi&=\sigma^{-1}c(\lambda)\partial_z^{-2}\Pro\SL\A_{\phi s}\label{eq:M1phi_tr}.
\end{align}

\subsection{The good unknown}
Before we proceed with the investigation of local bifurcation, we first introduce an isomorphism, which facilitates the computations later and is sometimes called $\T$-isomorphism in the literature (for example, in \cite{EEW2011, Varholm20}). The discovery of the importance of such a new variable (here $\theta$) goes back to Alinhac \cite{Alinhac89}, who called it the \enquote{good unknown} in a very general context, and Lannes \cite{Lannes05}, who introduced it in the context of water wave equations.

\begin{lemma}
	Let
	\[Y\coloneqq\left\{\theta\in C_{\per,\e}^{0,\alpha}(\overline{\Omega_0}):\SL\theta\in C_{0,\per,\e}^{2,\alpha}(\R),\ \I\theta\in H^1_{\per,\e}(\Omega_0^\I)\right\}\]
	and assume that $c(\lambda)\neq 0$. Then
	\[\T(\lambda)\colon Y\to X,\quad\T(\lambda)\theta=\left(-\frac{d\SL\theta}{c(\lambda)},\theta-\frac{2\psi^\lambda+s\psi_s^\lambda}{c(\lambda)}\SL\theta\right)\]
	is an isomorphism. Its inverse is given by
	\[[\T(\lambda)]^{-1}(\delta \eta,\delta\phi)=\delta\phi-\frac{2\psi^\lambda+s\psi_s^\lambda}{d}\delta\eta.\]
\end{lemma}
\begin{proof}
	Both $\T(\lambda)$ and $[\T(\lambda)]^{-1}$ are well-defined, and a simple computation shows that they are inverse to each other.
\end{proof}
Let us now consider a trivial solution $(\lambda,0,0)$. In view of the $\T$-isomorphism, we introduce
\[\LL(\lambda)\coloneqq [\F_{(\eta,\phi)}(\lambda,0,0)]\circ [\T(\lambda)]\colon Y\to X\]
whenever $c(\lambda)\neq 0$. Now recall that
\[\F_{(\eta,\phi)}=(\delta\eta-\M_\eta^1-\M_\phi^1,\delta\phi-\A_\eta-\A_\phi).\]
For given $\eta$ we denote by $V=V[\eta]$ the unique solution of
\[\Delta_d\I V=0\text{ in }\Omega_0^\I,\quad \I V=\eta\text{ on }|y|=1.\]
We notice that
\[-\frac{2\psi^\lambda+s\psi^\lambda_s}{c(\lambda)}\SL\theta+\frac{d}{c(\lambda)}\A_\eta\SL\theta+\A_\phi\left(\frac{2\psi^\lambda+s\psi^\lambda_s}{c(\lambda)}\SL\theta\right)=-V[\SL\theta].\]
Indeed, from
\begin{align*}
	\partial_{y_iy_i}\I(2\psi^\lambda+s\psi^\lambda_s)&=\partial_{y_iy_i}(2\I\psi^\lambda+y_j(\I\psi^\lambda)_{y_j})=2(\I\psi^\lambda)_{y_iy_i}+\partial_{y_j}(\I\psi^\lambda)_{y_j}+\partial_{y_i}(y_j(\I\psi^\lambda)_{y_iy_j})\\
	&=4(\I\psi^\lambda)_{y_iy_i}+y_j(\I\psi^\lambda)_{y_iy_iy_j}\nonumber\\
	&=4(\I\psi^\lambda)_{y_iy_i}-y_i\partial_{y_i}\left(d^2\gamma(d^2|y|^2\I\psi^\lambda)+\frac{1}{|y|^2}(FF')(d^2|y|^2\I\psi^\lambda)\right)\\
	&=4(\I\psi^\lambda)_{y_iy_i}-y_i\Bigg(d^4\gamma'(d^2|y|^2\I\psi^\lambda)(2y_i\I\psi^\lambda+|y|^2(\I\psi^\lambda)_{y_i})\\
	&\phantom{=\;}-\frac{2y_i}{|y|^4}(FF')(d^2|y|^2\I\psi^\lambda)+\frac{d^2}{|y|^2}(FF')'(d^2|y|^2\I\psi^\lambda)(2y_i\I\psi^\lambda+|y|^2(\I\psi^\lambda)_{y_i})\Bigg)\\
	&=4(\I\psi^\lambda)_{y_iy_i}+\frac{2}{|y|^2}(FF')(d^2|y|^2\I\psi^\lambda)\\
	&\phantom{=\;}-d^2\left(d^2|y|^2\gamma'(d^2|y|^2\I\psi^\lambda)+(FF')'(d^2|y|^2\I\psi^\lambda)\right)\left(2\I\psi^\lambda+y_i(\I\psi^\lambda)_{y_i}\right)
\end{align*}
we infer that the function $f\coloneqq-\frac{2\psi^\lambda+s\psi^\lambda_s}{c(\lambda)}\SL\theta+\frac{d}{c(\lambda)}\A_\eta\SL\theta+\A_\phi\left(\frac{2\psi^\lambda+s\psi^\lambda_s}{c(\lambda)}\SL\theta\right)+V[\SL\theta]$ satisfies
\begin{align*}
	\Delta_d\I f&=-\frac{2\I\psi^\lambda+y_i(\I\psi^\lambda)_{y_i}}{c(\lambda)}\SL\theta_{zz}-\frac{1}{d^2c(\lambda)}\Bigg(4(\I\psi^\lambda)_{y_iy_i}+\frac{2}{|y|^2}(FF')(d^2|y|^2\I\psi^\lambda)\\
	&\omit\hfill$\displaystyle-d^2\left(d^2|y|^2\gamma'(d^2|y|^2\I\psi^\lambda)+(FF')'(d^2|y|^2\I\psi^\lambda)\right)\left(2\I\psi^\lambda+y_i(\I\psi^\lambda)_{y_i}\right)\Bigg)\SL\theta$\\
	&\phantom{=\;}+\frac{d}{c(\lambda)}\left(\frac{4(\I\psi^\lambda)_{y_iy_i}}{d^3}\SL\theta+\frac{2\I\psi^\lambda+y_i(\I\psi^\lambda)_{y_i}}{d}\SL\theta_{zz}+\frac{2}{d^3|y|^2}(FF')(d^2|y|^2\I\psi^\lambda)\SL\theta\right)\\
	&\phantom{=\;}-\left(d^2|y|^2\gamma'(d^2|y|^2\I\psi^\lambda)+(FF')'(d^2|y|^2\I\psi^\lambda)\right)\frac{2\I\psi^\lambda+y_i(\I\psi^\lambda)_{y_i}}{c(\lambda)}\SL\theta\\
	&=0
\end{align*}
and $\I f=0$ at $|y|=1$. Thus, recalling \eqref{eq:Aeta_tr}, \eqref{eq:Aphi_tr}, \eqref{eq:M1eta_tr}, and \eqref{eq:M1phi_tr}, we can rewrite
\begin{align}\label{eq:L2}
	\LL_2(\lambda)\theta=\theta-(\A_\phi\theta+V[\SL\theta])
\end{align}
and
\begin{align}
	\LL_1(\lambda)\theta&=-\frac{d}{c(\lambda)}\SL\theta-\frac{1}{\sigma d^2c(\lambda)}(\sigma d+2d^2c(\lambda)^2+F(m(\lambda))^2)\partial_z^{-2}\SL\theta\nonumber\\
	&\phantom{=\;}+\sigma^{-1}c(\lambda)\partial_z^{-2}\Pro\SL\partial_s\left(\frac{d}{c(\lambda)}\A_\eta\SL\theta-\A_\phi\theta+\A_\phi\left(\frac{2\psi^\lambda+s\psi^\lambda_s}{c(\lambda)}\SL\theta\right)\right)\nonumber\\
	&=-\frac{d}{c(\lambda)}\SL\theta-\sigma^{-1}\left(\frac{\sigma}{dc(\lambda)}+2c(\lambda)+\frac{F(m(\lambda))^2}{d^2c(\lambda)}+(d^2\gamma+FF')(m(\lambda))\right)\partial_z^{-2}\SL\theta\nonumber\\
	&\phantom{=\;}-\sigma^{-1}c(\lambda)\partial_z^{-2}\Pro\SL\partial_s(\A_\phi\theta+V[\SL\theta])\label{eq:L1}
\end{align}
because of
\begin{align*}
	\SL\partial_s(2\psi^\lambda+s\psi^\lambda_s)&=\SL(3\psi_s^\lambda+s\psi^\lambda_{ss})=\SL\left(\frac3s\psi^\lambda_s+\psi^\lambda_{ss}\right)=\SL\left(-d^2\gamma(d^2s^2\psi^\lambda)-\frac{1}{s^2}(FF')(d^2s^2\psi^\lambda)\right)\\
	&=-(d^2\gamma+FF')(m(\lambda)).
\end{align*}
Notice that, under the assumption $\theta\in C_\per^{2,\alpha}(\overline{\Omega_0})$, $\LL_2(\lambda)\theta$ is the unique solution of
\begin{subequations}\label{eq:L2theta_rewritten}
\begin{align}
	\Delta_d[\I\LL_2(\lambda)\theta]&=\Delta_d\theta+\left(d^2|y|^2\gamma'(d^2|y|^2\I\psi^\lambda)+(FF')'(d^2|y|^2\I\psi^\lambda)\right)\I\theta&\text{in }\Omega_0^\I,\\
	\I\LL_2(\lambda)\theta&=0&\text{on }|y|=1,
\end{align}
\end{subequations}
and $\LL_1(\lambda)\theta$ is (in the set of $L$-periodic functions with zero average) uniquely determined by
\begin{align}
	[\LL_1(\lambda)\theta]_{zz}&=-\frac{d}{c(\lambda)}\SL\theta_{zz}-\sigma^{-1}\left(\frac{\sigma}{dc(\lambda)}+2c(\lambda)+\frac{F(m(\lambda))^2}{d^2c(\lambda)}+(d^2\gamma+FF')(m(\lambda))\right)\SL\theta\nonumber\\
	&\phantom{=\;}-\sigma^{-1}c(\lambda)\Pro\SL\partial_s(\A_\phi\theta+V[\SL\theta])\label{eq:L1theta_rewritten_2}\\
	&=-\frac{d}{c(\lambda)}\SL\theta_{zz}-\sigma^{-1}\left(\frac{\sigma}{dc(\lambda)}+2c(\lambda)+\frac{F(m(\lambda))^2}{d^2c(\lambda)}+(d^2\gamma+FF')(m(\lambda))\right)\SL\theta\nonumber\\
	&\phantom{=\;}-\sigma^{-1}c(\lambda)\Pro\SL\partial_s(\theta-\LL_2(\lambda)\theta).\label{eq:L1theta_rewritten}
\end{align}

\subsection{Kernel}
We now turn to the investigation of the kernel of $\F_{(\eta,\phi)}(\lambda,0,0)$. Clearly, in view of the $\T$-isomorphism it suffices to study the kernel of $\LL$; here and in the following, we will suppress the dependency of $\LL$ on $\lambda$. From \eqref{eq:L2} we infer that $\theta\in C^{2,\alpha}(\overline{\Omega_0})$ provided $\LL\theta=0$. Thus, combining \eqref{eq:L2theta_rewritten} and \eqref{eq:L1theta_rewritten} yields
\begin{align*}
	\LL\theta=0\;&\Longleftrightarrow\;\theta\in C_\per^{2,\alpha}(\overline{\Omega_0})\text{, and}\\
	&\Delta_d\I\theta+\left(d^2|y|^2\gamma'(d^2|y|^2\I\psi^\lambda)+(FF')'(d^2|y|^2\I\psi^\lambda)\right)\I\theta=0\text{, and}\\
	&\frac{d\SL\theta_{zz}}{c(\lambda)}+\sigma^{-1}\left(\frac{\sigma}{dc(\lambda)}+2c(\lambda)+\frac{F(m(\lambda))^2}{d^2c(\lambda)}+(d^2\gamma+FF')(m(\lambda))\right)\SL\theta+\sigma^{-1}c(\lambda)\Pro\SL\theta_s=0.
\end{align*}
Let us now write $\theta(s,z)=\sum_{k=0}^\infty\theta_k(s)\cos(k\nu z)$ as a Fourier series. Then we easily see that
\begin{align*}
	\LL\theta=0\quad\Longleftrightarrow\quad\eqref{eq:LLtheta0=0}\quad\text{and}\quad\forall k\ge 1:\eqref{eq:LLthetak=0},
\end{align*}
where
\begin{align}\label{eq:LLtheta0=0}
	\left(\frac{1}{d^2}\partial_{y_iy_i}+d^2|y|^2\gamma'(d^2|y|^2\I\psi^\lambda)+(FF')'(d^2|y|^2\I\psi^\lambda)\right)\I\theta_0=0,
\end{align}
noticing that $\theta_0(1)=0$ is already included in the definition of $Y$, and
\begin{subequations}\label{eq:LLthetak=0}
	\begin{align}
	\left(\frac{1}{d^2}\partial_{y_iy_i}+d^2|y|^2\gamma'(d^2|y|^2\I\psi^\lambda)+(FF')'(d^2|y|^2\I\psi^\lambda)-(k\nu)^2\right)\I\theta_k&=0,\label{eq:LL2thetak=0}\\
	\left(\frac{\sigma}{dc(\lambda)}(1-(k\nu)^2d^2)+2c(\lambda)+\frac{F(m(\lambda))^2}{d^2c(\lambda)}+(d^2\gamma+FF')(m(\lambda))\right)\theta_k(1)+c(\lambda)\partial_s\theta_k(1)&=0.\label{eq:LL1thetak=0}
	\end{align}
\end{subequations}
For $\mu\in\R$, let us now introduce the function $\tilde\beta=\tilde\beta^{\mu,\lambda}$, which is defined to be the unique solution of the singular Cauchy problem
\begin{align}\label{eq:tildebeta}
	\left(\partial_s^2+\frac3s\partial_s+d^4s^2\gamma'(d^2s^2\psi^\lambda)+d^2(FF')'(d^2s^2\psi^\lambda)+\mu d^2\right)\tilde\beta=0\text{ on }(0,1],\quad\tilde\beta_s(0)=0,\quad\tilde\beta(0)=1.
\end{align}
Indeed, this problem has a unique solution $\tilde\beta\in C^{2,\alpha}([0,1])$ by the same argument as in Section \ref{sec:trivial}. 

Henceforth, we shall assume that 
\begin{align}\label{ass:SL-spectrum}
	\tilde\beta^{0,\lambda}(1)\ne 0.
\end{align}
Thus, we see that \eqref{eq:LLtheta0=0} only has the trivial solution $\theta_0=0$. Indeed, if $\theta_0$ solves \eqref{eq:LLtheta0=0}, we have $\I\theta_0\in C^{2,\alpha}(\overline{\Omega_0})$, and therefore $\partial_s\theta_0(0)=0$. Hence, $\theta_0$ is a multiple of $\tilde\beta^{0,\lambda}$. But since necessarily $\theta_0(1)=0$, $\theta_0$ has to vanish identically in view of \eqref{ass:SL-spectrum}.

Let us now turn to $k\ge 1$ and notice as above that $\partial_s\theta_k(0)=0$ provided \eqref{eq:LL2thetak=0}. Thus, $\theta_k$ is a multiple of $\tilde\beta^{-(k\nu)^2,\lambda}$ if and only if \eqref{eq:LL2thetak=0} holds. First suppose that $\tilde\beta^{-(k\nu)^2,\lambda}(1)=0$ and that \eqref{eq:LLthetak=0} is satisfied. Then necessarily $\theta_k(1)=0$. Since therefore also $\partial_s\theta_k(1)=0$ by virtue of \eqref{eq:LL1thetak=0}, we conclude $\theta_k=0$. On the other hand, suppose that $\tilde\beta^{-(k\nu)^2,\lambda}(1)\ne 0$ and define $\beta^{-(k\nu)^2,\lambda}\coloneqq\tilde\beta^{-(k\nu)^2,\lambda}/\tilde\beta^{-(k\nu)^2,\lambda}(1)$. Hence, \eqref{eq:LLthetak=0} has a nontrivial solution $\theta_k$ if and only if the dispersion relation
\[\D(-(k\nu)^2,\lambda)=0,\]
where
\begin{align}\label{eq:d(k,lambda)}
	\D(\mu,\lambda)\coloneqq\beta_s^{\mu,\lambda}(1)+\frac{\sigma}{dc(\lambda)^2}(1+\mu d^2)+2+\frac{F(m(\lambda))^2}{d^2c(\lambda)^2}+\frac{(d^2\gamma+FF')(m(\lambda))}{c(\lambda)},
\end{align}
is satisfied, and in this case $\theta_k$ is a multiple of $\beta^{-(k\nu)^2,\lambda}$. We summarise our results concerning the kernel:
\begin{lemma}\label{lma:kernel}
	Given $\lambda\in\R$ with $c(\lambda)\ne 0$ and under the assumption \eqref{ass:SL-spectrum}, a function $\theta\in Y$, admitting the Fourier decomposition $\theta(s,z)=\sum_{k=0}^\infty\theta_k(s)\cos(k\nu z)$, is in the kernel of $\LL(\lambda)$ if and only if $\theta_0=0$ and for each $k\ge 1$
	\begin{enumerate}[label=(\alph*)]
		\item $\theta_k=0$, or
		\item $\tilde\beta^{-(k\nu)^2,\lambda}(1)\ne 0$, $\theta_k$ is a multiple of $\tilde\beta^{-(k\nu)^2,\lambda}$, and the dispersion relation
		\[\D(-(k\nu)^2,\lambda)=0\]
		holds, with $\D$ given in \eqref{eq:d(k,lambda)}.
	\end{enumerate}
\end{lemma}
\begin{remark}
	Clearly, $\D(\mu,\lambda)$ is at first only defined if $\tilde\beta^{\mu,\lambda}(1)\ne 0$. If this property fails to hold, we set $\D(\mu,\lambda)\coloneqq\infty$ in the following.
\end{remark}

\subsection{Range}
Before we proceed with the investigation of the transversality condition, we first prove that the range of $\LL$ can be written as an orthogonal complement with respect to a suitable inner product. This will be helpful later. To this end, we introduce the inner product
\[\langle(f_1,g_1),(f_2,g_2)\rangle\coloneqq\langle f_1',f_2'\rangle_{L^2([0,L])}+\langle\nabla_d \I g_1,\nabla_d \I g_2\rangle_{L^2(\tilde\Omega_0^\I)}\]
for $f_1,f_2\in H^1_{0,\per}(\R)$, $g_1,g_2\colon\tilde\Omega_0\to\R$ with $\I g_1,\I g_2\in H^1_\per(\tilde\Omega_0^\I)$, where $\tilde\Omega_0\coloneqq[0,1)\times(0,L)$ is one periodic instance of $\Omega_0$ and $\nabla_d\coloneqq(\partial_{y_1}/d,\ldots,\partial_{y_4}/d,\partial_z)^T$; in order to avoid misunderstanding, we point out that the index \enquote{0} in $H^1_{0,\per}(\R)$ means \enquote{zero average} as before and not \enquote{zero boundary values}. This inner product is positive definite on the space

\[H^1_{0,\per}(\R)\times\left\{g\colon\tilde\Omega_0\to\R:\I g\in H^1_\per(\tilde\Omega_0^\I),\langle\SL g\rangle=0\right\}.\]
Notice that
\[\langle f_1',f_2'\rangle_{L^2([0,L])}=-\langle f_1,f_2''\rangle_{L^2([0,L])}\]
if $f_2\in H^2_\per(\R)$ and that
\[\langle\nabla_d g_1,\nabla_d g_2\rangle_{L^2(\tilde\Omega_0^\I)}=-\langle g_1,\Delta_d g_2\rangle_{L^2(\tilde\Omega_0^\I)}+\frac{2\pi^2}{d^2}\langle\SL g_1,\SL\partial_s g_2\rangle_{L^2([0,L])}\]
if $\I g_2\in H^2_\per(\tilde\Omega_0)$, using that $2\pi^2$ is the surface area of the $3$-sphere.

Using \eqref{eq:L2}, \eqref{eq:L2theta_rewritten}, and \eqref{eq:L1theta_rewritten_2} we now compute for smooth $\theta,\vartheta\in Y$
\begin{align*}
&\left\langle\left(\frac{2\pi^2\sigma}{d^2c(\lambda)}\SL\theta,\theta\right),\LL\vartheta\right\rangle\\
&=\frac{2\pi^2\sigma}{d^2c(\lambda)}\Bigg\langle\SL\theta,\frac{d}{c(\lambda)}\SL\vartheta_{zz}+\sigma^{-1}\left(\frac{\sigma}{dc(\lambda)}+2c(\lambda)+\frac{F(m(\lambda))^2}{d^2c(\lambda)}+(d^2\gamma+FF')(m(\lambda))\right)\SL\vartheta\\
&\omit\hfill$\displaystyle+\sigma^{-1}c(\lambda)\Pro\SL\partial_s(\A_\phi\vartheta+V[\SL\vartheta])\Bigg\rangle_{L^2([0,L])}$\\
&\phantom{=\;}-\left\langle\I\theta,\Delta_d\I\vartheta+\left(d^2|y|^2\gamma'(d^2|y|^2\I\psi^\lambda)+(FF')'(d^2|y|^2\I\psi^\lambda)\right)\I\vartheta\right\rangle_{L^2(\tilde\Omega_0^\I)}\\
&\phantom{=\;}+\frac{2\pi^2}{d^2}\langle\SL\theta,\SL\partial_s(\vartheta-(\A_\phi\vartheta+V[\SL\vartheta]))\rangle_{L^2([0,L])}\\
&=-\frac{2\pi^2\sigma}{dc(\lambda)^2}\langle\SL\theta_z,\SL\vartheta_z\rangle_{L^2([0,2\pi])}+\langle\nabla_d\I\theta,\nabla_d\I\vartheta\rangle_{L^2(\tilde\Omega_0^\I)}\\
&\phantom{=\;}+\frac{2\pi^2}{d^2c(\lambda)}\left(\frac{\sigma}{dc(\lambda)}+2c(\lambda)+\frac{F(m(\lambda))^2}{d^2c(\lambda)}+(d^2\gamma+FF')(m(\lambda))\right)\langle\SL\theta,\SL\vartheta\rangle_{L^2([0,L])}\\
&\phantom{=\;}-\left\langle\I\theta,\left(d^2|y|^2\gamma'(d^2|y|^2\I\psi^\lambda)+(FF')'(d^2|y|^2\I\psi^\lambda)\right)\I\vartheta\right\rangle_{L^2(\tilde\Omega_0^\I)}
\end{align*}
making use of $\langle\SL\theta\rangle=0$. Noticing that the terms at the beginning and at the end of this computation only involve at most first derivatives of $\theta$ and $\vartheta$, an easy approximation argument shows that this relation also holds for general $\theta,\vartheta\in Y$. Moreover, since the last expression is symmetric in $\theta$ and $\vartheta$, we can also go in the opposite direction with reversed roles and arrive at the symmetry property
\[\left\langle\left(\frac{2\pi^2\sigma}{d^2c(\lambda)}\SL\theta,\theta\right),\LL\vartheta\right\rangle=\left\langle\LL\theta,\left(\frac{2\pi^2\sigma}{d^2c(\lambda)}\SL\vartheta,\vartheta\right)\right\rangle.\]
Thus, the range of $\LL$ is the orthogonal complement of
\[\left\{\left(\frac{2\pi^2\sigma}{d^2c(\lambda)}\SL\theta,\theta\right):\theta\in\ker\LL\right\}\]
with respect to $\langle\cdot,\cdot\rangle$. Indeed, one inclusion is an immediate consequence of the symmetry property and the other inclusion follows from the facts that we already know that $\LL$, being a compact perturbation of the identity, is Fredholm with index zero and that $\LL$ gains no additional kernel when extended to functions $\theta$ of class $H^1$.

\subsection{Transversality condition}
Assuming that the kernel is spanned by the function $\theta(s,z)=\beta^{-(k\nu)^2,\lambda}(s)\cos(k\nu z)$, we have to investigate whether $\LL_\lambda\theta$ is not in the range of $\LL$, which is equivalent to
\[\left\langle\left(\frac{2\pi^2\sigma}{d^2c(\lambda)}\SL\theta,\theta\right),\LL_\lambda\theta\right\rangle\neq 0\]
by the preceding considerations. Differentiating \eqref{eq:L2} and \eqref{eq:L1} with respect to $\lambda$, for general $\theta$ it holds
\begin{align*}
	\LL_{\lambda,1}\theta&=-\partial_\lambda\left(\frac{d}{c(\lambda)}\right)\SL\theta-\sigma^{-1}\partial_\lambda\left(\frac{\sigma}{dc(\lambda)}+2c(\lambda)+\frac{F(m(\lambda))^2}{d^2c(\lambda)}+(d^2\gamma+FF')(m(\lambda))\right)\partial_z^{-2}\SL\theta\\
	&\phantom{=\;}-\sigma^{-1}c_\lambda(\lambda)\partial_z^{-2}\Pro\SL\partial_s(\A_\phi\theta+V[\SL\theta])-\sigma^{-1}c(\lambda)\partial_z^{-2}\Pro\SL\partial_s\A_{\phi\lambda}\theta\\
	\LL_{\lambda,2}\theta&=-\A_{\phi\lambda}\theta,
\end{align*}
where $\A_{\phi\lambda}\theta$ is the unique solution of
\begin{align*}
	\Delta_d(\I\A_{\phi\lambda}\theta)&=-d^2|y|^2\left(d^2|y|^2\gamma''(d^2|y|^2\I\psi^\lambda)+(FF')''(d^2|y|^2\I\psi^\lambda)\right)\partial_\lambda\I\psi^\lambda\I\theta&\text{in }\Omega_0^\I,\\
	\I\A_{\phi\lambda}\theta&=0&\text{on }|y|=1.
\end{align*}
Thus, we have
\begin{align*}
&\left\langle\left(\frac{2\pi^2\sigma}{d^2c(\lambda)}\SL\theta,\theta\right),\LL_\lambda\theta\right\rangle\\
&=\frac{2\pi^2\sigma}{d^2c(\lambda)}\Bigg\langle\SL\theta,\partial_\lambda\left(\frac{d}{c(\lambda)}\right)\SL\theta_{zz}+\sigma^{-1}\partial_\lambda\left(\frac{\sigma}{dc(\lambda)}+2c(\lambda)+\frac{F(m(\lambda))^2}{d^2c(\lambda)}+(d^2\gamma+FF')(m(\lambda))\right)\SL\theta\\
&\omit\hfill$\displaystyle+\sigma^{-1}c_\lambda(\lambda)\Pro\SL\partial_s(\A_\phi\theta+V[\SL\theta])+\sigma^{-1}c(\lambda)\Pro\SL\partial_s\A_{\phi\lambda}\theta\Bigg\rangle_{L^2([0,L])}$\\
&\phantom{=\;}-\langle\I\theta,\Delta_d(-\I\A_{\phi\lambda}\theta)\rangle_{L^2(\tilde\Omega_0^\I)}+\frac{2\pi^2}{d^2}\langle\SL\theta,\SL\partial_s(-\A_{\phi\lambda}\theta)\rangle_{L^2([0,L])}\\
&=\frac{2\pi^2}{d^2}\left\langle\SL\theta,\partial_\lambda\left(\frac{\sigma d}{c(\lambda)^2}\right)\SL\theta_{zz}+\partial_\lambda\left(\frac{\sigma}{dc(\lambda)^2}+2+\frac{F(m(\lambda))^2}{d^2c(\lambda)^2}+\frac{(d^2\gamma+FF')(m(\lambda))}{c(\lambda)}\right)\SL\theta\right\rangle_{L^2([0,L])}\\
&\phantom{=\;}-\left\langle\I\theta,d^2|y|^2\left(d^2|y|^2\gamma''(d^2|y|^2\I\psi^\lambda)+(FF')''(d^2|y|^2\I\psi^\lambda)\right)\partial_\lambda\I\psi^\lambda\I\theta\right\rangle_{L^2(\tilde\Omega_0^\I)}
\end{align*}
whenever $\LL_1\theta=0$. Now let $\theta(s,z)=\beta^{-(k\nu)^2,\lambda}(s)\cos(k\nu z)$ and notice that $f=\partial_\lambda\beta^{-(k\nu)^2,\lambda}$ solves
\begin{align*}
	(\I f)_{y_iy_i}+\left(d^4|y|^2\gamma'(d^2|y|^2\psi^\lambda)+d^2(FF')'(d^2|y|^2\psi^\lambda)-(k\nu)^2d^2\right)\I f&\\
	=-d^4|y|^2\left(d^2|y|^2\gamma''(d^2|y|^2\I\psi^\lambda)+(FF')''(d^2|y|^2\I\psi^\lambda)\right)\I\beta^{-(k\nu)^2,\lambda}\partial_\lambda\I\psi^\lambda&&\text{in }|y|<1,\\
	\I f=0&&\text{on }|y|=1.
\end{align*}
Therefore,
\begin{align*}
	&\frac{d^2}{L\pi^2}\left\langle\left(\frac{2\pi^2\sigma}{d^2c(\lambda)}\SL\theta,\theta\right),\LL_\lambda\theta\right\rangle\\
	&=\frac{1}{2\pi^2}\int_{|y|<1}\I\beta^{-(k\nu)^2,\lambda}\Big((\I\partial_\lambda\beta^{-(k\nu)^2,\lambda})_{y_iy_i}\\
	&\phantom{=\frac{1}{2\pi^2}\int_{|y|<1}\I\beta^{-(k\nu)^2,\lambda}\Big(\;}+\left(d^4|y|^2\gamma'(d^2|y|^2\psi^\lambda)+d^2(FF')'(d^2|y|^2\psi^\lambda)-(k\nu)^2d^2\right)\I\partial_\lambda\beta^{-(k\nu)^2,\lambda}\Big)\,dy\\
	&\phantom{=\;}+\partial_\lambda\left(\frac{\sigma}{dc(\lambda)^2}(1-(k\nu)^2d^2)+2+\frac{F(m(\lambda))^2}{d^2c(\lambda)^2}+\frac{(d^2\gamma+FF')(m(\lambda))}{c(\lambda)}\right)\\
	&=\partial_\lambda\beta_s^{-(k\nu)^2,\lambda}(1)+\partial_\lambda\left(\frac{\sigma}{dc(\lambda)^2}(1-(k\nu)^2d^2)+2+\frac{F(m(\lambda))^2}{d^2c(\lambda)^2}+\frac{(d^2\gamma+FF')(m(\lambda))}{c(\lambda)}\right)
\end{align*}
after integrating by parts. Thus, we have proved:
\begin{lemma}\label{lma:transversality_condition}
	Given $\lambda\in\R$ with $c(\lambda)\ne 0$ and assuming that the kernel of $\LL(\lambda)$ is one-dimensional spanned by $\theta(s,z)=\beta^{-(k\nu)^2,\lambda}(s)\cos(k\nu z)$ for some $k\ge 1$, the transversality condition
	\[\LL_\lambda(\lambda)\theta\notin\im\LL(\lambda)\]
	is equivalent to
	\[\D_\lambda(-(k\nu)^2,\lambda)\neq 0,\]
	with $\D$ given in \eqref{eq:d(k,lambda)}.
\end{lemma}

\subsection{Result on local bifurcation}
We summarise our result of this section using the following local bifurcation theorem by Crandall--Rabinowitz \cite[Thm. I.5.1]{Kielhoefer}.
\begin{theorem}\label{thm:CrandallRabinowitz}
	Let $X$ be a Banach space, $U\subset\R\times X$ open, and $\F\colon U\to X$ have the property $\F(\cdot,0)=0$. Assume that there exists $\lambda_0\in\R$ such that $\F$ is of class $C^2$ in an open neighbourhood of $(\lambda_0,0)$, and suppose that $\F_x(\lambda_0,0)$ is a Fredholm operator with index zero and one-dimensional kernel spanned by $x_0\in X$, and that the transversality condition $\F_{\lambda x}(\lambda_0,0)x_0\notin\im \F_x(\lambda_0,0)$ holds. Then there exists $\varepsilon>0$ and a $C^1$-curve $(-\varepsilon,\varepsilon)\ni t\mapsto(\lambda^t,x^t)$ with $(\lambda^0,x^0)=(\lambda_0,0)$ and $x^t\neq 0$ for $t\neq 0$, and $\F(\lambda^t,x^t)=0$. Moreover, all solutions of $\F(\lambda,x)=0$ in a neighbourhood of $(\lambda_0,0)$ are on this curve or are trivial. Furthermore, the curve admits the asymptotic expansion $x^t=tx_0+o(t)$.
\end{theorem}
Applied to our problem, we obtain the following result.
\begin{theorem}\label{thm:LocalBifurcation}
	Assume \eqref{ass:SL-spectrum} and that there exists $\lambda_0\in\R$ with $c(\lambda_0)\ne 0$ such that the dispersion relation
	\[\D(-(k\nu)^2,\lambda_0)=0,\]
	with $\D$ given by \eqref{eq:d(k,lambda)}, has exactly one solution $k_0\in\N$ and assume that the transversality condition
	\[\D_\lambda(-(k_0\nu)^2,\lambda_0)\neq 0\]
	holds. Then there exists $\varepsilon>0$ and a $C^1$-curve $(-\varepsilon,\varepsilon)\ni t\mapsto(\lambda^t,\eta^t,\phi^t)$ with $(\lambda^0,\eta^0,\phi^0)=(\lambda_0,0,0)$, $\eta^t\neq 0$ for $t\neq 0$, and $\F(\lambda^t,\eta^t,\phi^t)=0$. Moreover, all solutions of $\F(\lambda,\eta,\phi)=0$ in a neighbourhood of $(\lambda_0,0,0)$ are on this curve or are trivial. Furthermore, the curve admits the asymptotic expansion $(\eta^t,\phi^t)=t\T(\lambda_0)\theta+o(t)$, where
	\begin{align*}
		\theta(s,z)&=\beta^{-(k_0\nu)^2,\lambda_0}(s)\cos(k_0\nu z),\\
		[\T(\lambda_0)\theta](x,y)&=\left(-\frac{d}{c(\lambda_0)},\beta^{-(k_0\nu)^2,\lambda_0}(s)-\frac{2\psi^{\lambda_0}(s)+s\psi_s^{\lambda_0}(s)}{c(\lambda_0)}\right)\cos(k_0\nu z).
	\end{align*}
\end{theorem}
\begin{proof}
	It is straightforward to apply Theorem \ref{thm:CrandallRabinowitz} in view of Lemmas \ref{lma:M_prop}, \ref{lma:kernel}, and \ref{lma:transversality_condition}, noticing that $F_{(\eta,\phi)}(\lambda_0,0,0)$ coincides with $\LL(\lambda_0)$ up to the isomorphism $\T(\lambda_0)$. Moreover, the asymptotic expansion tells us that $\eta(t)\neq 0$ after possibly shrinking $\varepsilon$.
\end{proof}

\section{Conditions for local bifurcation} \label{sec:conditions}
\subsection{Spectral properties}
In view of the defining equation \eqref{eq:tildebeta} for $\tilde\beta^{\mu,\lambda}$ and the dispersion relation $\D(\mu,\lambda)=0$ and writing $\varphi=\tilde\beta^{\mu,\lambda}$, we study the eigenvalue problem
\begin{subequations}\label{eigenvalue_problem}
	\begin{align}
	-d^{-2}s^{-3}(s^3\varphi')'+q^\lambda\varphi&=\mu\varphi&\text{in }(0,1),\\
	-g(\lambda)\varphi'(1)-h(\lambda)\varphi(1)&=\mu\varphi(1),&\label{eigenvalue_problem_BC}
	\end{align}
\end{subequations}
which is a singular Sturm-Liouville problem on $(0,1)$. Here and in the following, we denote
\begin{align*}
	q^\lambda(s)&\coloneqq-d^2s^2\gamma'(d^2s^2\psi^\lambda(s))-(FF')'(d^2s^2\psi^\lambda(s)),\\
	g(\lambda)&\coloneqq\sigma^{-1}d^{-1}c(\lambda)^2>0,\\
	h(\lambda)&\coloneqq d^{-2}+\sigma^{-1}d^{-1}c(\lambda)\left(2c(\lambda)+d^{-2}F(m(\lambda))^2+(d^2\gamma+FF')(m(\lambda))\right).
\end{align*}
Notice that we left out the condition $\tilde\beta^{\mu,\lambda}_s(0)=0$ in view of Lemma \ref{regularity_lemma} below. We first introduce the operators $T$ and $\tau$ via
\[D(T)=D(\tau)=\{\varphi\in L^2_{s^3}(0,1):\varphi,s^3\varphi'\in\text{AC}_{\text{loc}}(0,1],s^{-3}(s^3\varphi')'\in L^2_{s^3}(0,1)\}\]
and
\[T\varphi=-d^{-2}s^{-3}(s^3\varphi')',\quad\tau\varphi=T\varphi+q^\lambda\varphi,\quad\varphi\in D(\tau).\]
We collect some important properties of $T$, $\tau$, and $D(T)$:
\begin{lemma}\label{Ttau_lemma}
	The following holds:
	\begin{enumerate}[label=(\roman*)]
		\item The operators $T$ and $\tau$ are of limit point type at $0$ (and of regular type at $1$).
		\item For any $\varphi,\chi\in D(T)$ we have \[\lim_{s\to 0}\left(s^3\varphi'(s)\overline\chi(s)-\varphi(s)s^3\overline\chi'(s)\right)=0.\] In particular, \[\lim_{s\to 0}s^3\varphi(s)=\lim_{s\to 0}s^3\varphi'(s)=0.\]
	\end{enumerate}
\end{lemma}
\begin{proof}
	It is easy to see that $T$ is of limit point type at $0$, since $\varphi(s)=s^{-2}\notin L^2_{s^3}(0,1)$ solves $T\varphi=0$. Since $q^\lambda\in L^\infty(0,1)$, $\tau$ is also of limit point type at $0$ according to \cite[Corollary 7.4.1]{Zettl}. Thus, (i) is proved. As for (ii), the first statement is an immediate consequence of $T$ being of limit point type at $0$; see \cite[Lemmas 10.2.3, 10.4.1(b)]{Zettl}. Plugging in $\chi(s)=1$ and then $\chi(s)=s$ (which both belong to $D(T)$) yields the second statement.
\end{proof}
As a consequence the following result holds; in particular, this explains why we could leave out $\phi'(0)=0$ in \eqref{eigenvalue_problem}.
\begin{lemma}\label{regularity_lemma}
	Let $q,f\in C^{0,\alpha}([0,1])$ (or, equivalently, $\I q,\I f\in C^{0,\alpha}(\overline{B_1(0)})$) and $\varphi\in D(T)$ satisfy \[T\varphi=q\varphi+f.\] Then, $\I\varphi\in C^{2,\alpha}(\overline{B_1(0)})$ and solves \[\Delta_4\I\varphi=\I q\I\varphi+\I f.\] Obviously, the converse also holds. Moreover, in this case $\varphi\in C^{2,\alpha}([0,1])$ and $\varphi'(0)=0$.
\end{lemma}
\begin{proof}
	Clearly, $\I\varphi$ has weak derivatives on $B_1(0)\setminus\{0\}$; in particular, $\nabla_4\I\varphi=\varphi'e_s$ a.e. First, we claim that this also holds on $B_1(0)$. To this end, we first note that $\I\varphi\in L^2(B_1(0))$ due to $\varphi\in L^2_{s^3}(0,1)$. Now fix $v\in C_c^\infty(B_1(0);\R^4)$ and let $\varepsilon>0$. We have to pass to the limit $\varepsilon\to 0$ in the identity \[\int_{\varepsilon\le|y|\le1}\nabla_4\I\varphi\cdot v\,dy=-\int_{\varepsilon\le|y|\le1}\I\varphi\nabla_4\cdot v\,dy-\int_{|y|=\varepsilon}\I\varphi v\cdot e_s\,dS_y;\] note that the surface integral is well-defined since $\varphi\in\text{AC}_{\text{loc}}(0,1]$. Passing to the limit in the volume integrals is easy, as $|\nabla_4\I\varphi\cdot v|\le|\varphi'||v|$, $|\I\varphi\nabla_4\cdot v|\le|\varphi||\nabla_4 v|$, and $s^3\varphi,s^3\varphi'\in L^\infty(0,1)$ due to Lemma \ref{Ttau_lemma}(ii). Also because of Lemma \ref{Ttau_lemma}(ii) the surface integral vanishes in the limit, since its modulus can be estimated by $C\varepsilon^3|\varphi(\varepsilon)|$, where $C>0$ only depends on $\|v\|_\infty$.
	
	The next step is to show that $\I\varphi$ solves $\Delta_4\I\varphi=\I q\I\varphi+\I f$ on $B_1(0)$ in the weak sense. Clearly, we infer from the preceding considerations that $\I\varphi\in W^{1,1}(B_1(0))$. For fixed $v\in C_c^\infty(B_1(0))$ it holds that \[-\int_{B_1(0)}\nabla_4\I\varphi\cdot\nabla_4 v\,dy=-\int\int_0^1\varphi'v_ss^3\,ds\,d\Omega=\int\int_0^1(q\varphi+f)vs^3\,ds\,d\Omega=\int_{B_1(0)}(\I q\I\varphi+\I f)v\,dy,\] where $\int\cdots d\Omega$ denotes the integration with respect to the three angles in spherical coordinates of $\R^4$. It is very important to notice that here no boundary terms at $s=0$ appear although $v$ does not have to vanish there. This is due to the fact that $\lim_{s\to 0}s^3\varphi'(s)=0$ (see Lemma \ref{Ttau_lemma}(ii)), so the weak form \[-\int_0^1\varphi'w's^3\,ds=\int_0^1(q\varphi+f)ws^3\,ds\] also applies for smooth functions $w$ on $[0,1]$ having support at $s=0$ (but not at $s=1$).
	
	Finally, we infer from elliptic regularity that $\I\varphi\in C^{2,\alpha}(\overline{B_1(0)})$. Indeed, since $\Delta_4\I\varphi=\I q\I\varphi+\I f\in L^2(B_1(0))$, we have $\I\varphi\in H^2(B_1(0))\subset L^p(B_1(0))$, $1\le p<\infty$. Thus, $\Delta_4\I\varphi\in L^p(B_1(0))$ and $\I\varphi\in W^{2,p}(B_1(0))\subset C^{0,\alpha}(\overline{B_1(0)})$ for $p$ large. Hence,  $\Delta_4\I\varphi\in C^{0,\alpha}(\overline{B_1(0)})$ and therefore $\I\varphi\in C^{2,\alpha}(\overline{B_1(0)})$. The remaining statements clearly hold true.
\end{proof}

To introduce a functional-analytic setting when also taking the boundary condition \eqref{eigenvalue_problem_BC} into account, we let $H=L^2_{s^3}(0,1)\times\C$. In the following, we write elements $u\in H$ as $u=(\varphi,b)$. Equipped with the indefinite inner product \[[u_1,u_2]=\langle d^2\varphi_1,\varphi_2\rangle_{L^2_{s^3}}-g(\lambda)^{-1}b_1\overline{b_2},\] $H$ becomes a Pontryagin $\pi_1$-space. Furthermore, we introduce the operator $K$ given by \[D(K)=\{u\in H:\varphi\in D(\tau),b=\varphi(1)\}\] and \[Ku=\left(\tau\varphi,-g(\lambda)\varphi'(1)-h(\lambda)\varphi(1)\right),\quad u\in D(K),\] which is clearly densely defined. Observe that the eigenvalues (-functions) of $K$ are exactly the eigenvalues (-functions) of \eqref{eigenvalue_problem}. We have the following.
\begin{lemma}
	$K$ is self-adjoint.
\end{lemma}
\begin{proof}
	We first prove that $K$ is symmetric. To this end, for $u_1,u_2\in H$, $x\in (0,1)$ let \[[u_1,u_2]_x\coloneqq\langle d^2\varphi_1,\varphi_2\rangle_{L^2_{s^3}(x,1)}-g(\lambda)^{-1}b_1\overline{b_2}.\] Now if $u_1,u_2\in D(K)$ we have, after integrating by parts,
	\[[Ku_1,u_2]_x-[u_1,Ku_2]_x=x^3\varphi_1'(x)\overline{\varphi_2}(x)-\varphi_1(x)x^3\overline{\varphi_2}'(x).\] Clearly, $K$ is symmetric if and only if the first expression converges to $0$ as $x\to 0$ (for any $u_1,u_2\in D(K)$). But the second expression converges to $0$ due to Lemma \ref{Ttau_lemma}(ii).

	To see that $K$ is even self-adjoint, we first note that obviously $H$ admits the fundamental decomposition $H=(L^2_{s^3}(0,1)\times\{0\})\dot+(\{0\}\times\C)$ into a positive and negative subspace. Associated to this decomposition is the fundamental symmetry \[J=\begin{pmatrix}\mathrm{id}&0\\0&-1\end{pmatrix}\] and the Hilbert inner product $\langle u_1,u_2\rangle_J=[Ju_1,u_2]=\langle d^2\varphi_1,\varphi_2\rangle_{L^2_{s^3}}+g(\lambda)^{-1}b_1\overline{b_2}$. The operator $JK$ is self-adjoint with respect to $\langle\cdot,\cdot\rangle_J$, since now the assumptions of \cite[Theorem 1]{Hinton} are satisfied. In particular, denoting the $J$-adjoint by an upper index $\langle*\rangle$, we have \[D(K)=D(JK)=D\left((JK)^{\langle*\rangle}\right)=D\left(K^{\langle*\rangle}J^{\langle*\rangle}\right)=D\left(K^{\langle*\rangle}J\right)=D(JK^*)=D(K^*),\] as $J^{\langle*\rangle}=J$ and $JK^{\langle*\rangle}J=K^*$ (cf. \cite[Lemma VI.2.1]{Bognar}). Since $K$ is already known to be symmetric, the proof is complete.
\end{proof}
Now we can prove the following important result.
\begin{proposition}
	The spectrum of $K$ is purely discrete and consists of only (geometrically) simple eigenvalues.
\end{proposition}
\begin{proof}
	Following the proof of \cite[Theorem 2]{Hinton} using the $J$-norm $\|u\|_J=\sqrt{\langle u,u\rangle_J}$, we see that the essential spectra of $K$ and $\tau$ coincide. Notice that the criterion \cite[Theorem XIII.7.1]{DunfordSchwartz} applied there is purely topological and does not make use of an additional structure from an (definite or indefinite) inner product. To see that the essential spectrum of $\tau$ is empty, we can apply a criterion of \cite{Friedrichs}; see also \cite{Hinton}. Indeed, $q^\lambda$ is obviously bounded from below on $(0,1)$ and moreover \[\lim_{s\to 0}\left(q^\lambda(s)+\frac{1}{4d^2s^6\left(\int_s^1\sigma^{-3}\,d\sigma\right)^2}\right)=\lim_{s\to 0}d^{-2}s^{-2}(s^2-1)^{-2}=\infty.\] Finally, it is a priori clear that each eigenvalue of $K$ cannot have (geometric) multiplicity larger than two; the case of multiplicity two is excluded by the fact that $\tau$ is of limit point type at $0$.
\end{proof}
In fact, we can say more about the location of the eigenvalues of $K$. To this end, the following Lemma turns out to be useful.
\begin{lemma}\label{lma:Ku,u}
	For any $u\in D(K)$ we have
	\[[Ku,u]=\|\varphi'\|_{L^2_{s^3}}^2+\int_0^1 d^2s^3q^\lambda|\varphi|^2\,ds+\frac{h(\lambda)}{g(\lambda)}|\varphi(1)|^2.\]
\end{lemma}
\begin{proof}
	The only critical point is to ensure that no boundary terms at $0$ appear after an integration by parts, which again follows from Lemma \ref{Ttau_lemma}(ii).
\end{proof}
\begin{proposition}
	$K$ has no or exactly two nonreal eigenvalues, and in the latter case they are the complex conjugate of each other. Moreover, the (real part of the) spectrum of $K$ is bounded from below.
\end{proposition}
\begin{proof}
	The first assertion is clear since $H$ is a $\pi_1$-space and $K$ is self-adjoint; cf. \cite{IKL}. To prove the second statement, we use a perturbation argument. First notice that $q^\lambda$ does not affect the domain of the associated operator. Now let $K_0$ be the operator in the case $\gamma=F=0$, which yields $q^\lambda=0$ and $h(\lambda)>0$. By Lemma \ref{lma:Ku,u} we have $[K_0u,u]>0$ if $u\neq 0$. Thus, there exists exactly one negative eigenvector of $K_0$; cf. again \cite{IKL}. Therefore, $K_0$ has exactly one negative eigenvalue $\mu_0$ and its other eigenvalues are positive. With the same proof as in \cite[Lemma 3]{Wahlen06} we conclude that for some constant $C>0$ the estimate
	\[\|(K_0-\mu I)^{-1}\|_J\le\frac{C}{|\mu-\mu_0|},\quad\mu\in(-\infty,\mu_0),\]
	for the resolvent holds. If $\gamma$ and $F$ are arbitrary, we define the perturbation $A$ via $D(A)=\{u\in H:\varphi\in D(T),b=\varphi(1)\}$ and \[Au=\left(q^\lambda\varphi,-\sigma d^{-1}c(\lambda)\left(d^{-2}F(m(\lambda))^2+(d^2\gamma+FF')(m(\lambda))\right)\varphi(1)\right),\quad u\in D(A).\]
	Clearly, $A$ is densely defined and bounded, and we have $K=K_0+A$. Now consider a real $\mu<\mu_0-C\|A\|_J$. Because of \[K-\mu I=\left(I+A(K_0-\mu I)^{-1}\right)(K_0-\mu I)\] and \[\left\|A(K_0-\mu I)^{-1}\right\|_J\le\|A\|_J\cdot\frac{C}{|\mu-\mu_0|}<1,\] the resolvent operator $K-\mu I$ is invertible in view of the Neumann series. This completes the proof.
\end{proof}
Under a certain condition we can infer even more properties of the spectrum of $K$, as we see in what follows.
\begin{proposition}
	Assume that 
	\begin{align}\label{ass:realspec_algsimple}
	h(\lambda)>\|q^\lambda_-\|_\infty
	\end{align}
	where $q^\lambda_-$ denotes the negative part of $q^\lambda$. Then the operator $K$ has only real eigenvalues, $K$ has exactly one eigenvalue $\mu<-\|q^\lambda_-\|_\infty$, and all its other eigenvalues satisfy $\mu>-\|q^\lambda_-\|_\infty$. Moreover, all eigenvalues are algebraically simple.
\end{proposition}
\begin{proof}
	Let $\mu$ be an eigenvalue of $K$ and $u=(\varphi,\varphi(1))$ an associated eigenvector. Due to Lemma \ref{lma:Ku,u} we can calculate
	\begin{align*}
	\mu[u,u]&=[Ku,u]=\|\varphi'\|_{L^2_{s^3}}^2+\int_0^1 d^2s^3q^\lambda|\varphi|^2\,ds+\frac{h(\lambda)}{g(\lambda)}|\varphi(1)|^2\ge -d^2\|q^\lambda_-\|_\infty\|\varphi\|_{L^2_{s^3}}^2+\frac{h(\lambda)}{g(\lambda)}|\varphi(1)|^2\\
	&=-\|q^\lambda_-\|_\infty[u,u]+\frac{h(\lambda)-\|q^\lambda_-\|_\infty}{g(\lambda)}|\varphi(1)|^2.
	\end{align*}
	By assumption and since $\varphi(1)\neq 0$ (otherwise, also $\varphi'(1)=0$ and thus $\varphi\equiv 0$), it follows that
	\[(\mu+\|q^\lambda_-\|_\infty)[u,u]>0.\] Hence, $u$ cannot be neutral and $\mu$ has to be real. Since, additionally, by \cite{IKL} -- noting that $H$ is a $\pi_1$-space -- there exists exactly one nonpositive eigenvector of $K$, the first assertion follows immediately. The second statement is a direct consequence of the fact that all eigenvalues are real and no eigenvectors are neutral.
\end{proof}
\begin{remark}
	If \eqref{ass:realspec_algsimple} holds, then the assumptions of the next lemmas are satisfied. Moreover, we will discuss \eqref{ass:realspec_algsimple} later when looking at specific examples. Physically speaking, \eqref{ass:realspec_algsimple} is satisfied if the wave speed of the trivial solution at the surface is large compared to the angular components of the velocity and the vorticity (which depend on $\lambda$); more precisely, if
	\begin{align*}
		|c(\lambda)|&\notin[c_-,c_+],\\
		c_\pm&\coloneqq\frac14\left(d\omega^\vartheta-(u^\vartheta)^2\pm\sqrt{(d\omega^\vartheta-(u^\vartheta)^2)^2+8\sigma d(d^2\|\gamma'\|_\infty+\|(FF')'\|_\infty-d^{-2})}\right)
	\end{align*}
	(where the condition is regarded to be vacuous if $c_\pm$ are not real). In particular, if $\gamma$ and $FF'$ are bounded, this condition is satisfied if \enquote{$c(\lambda)$ is sufficiently large} or, provided additionally $F$ is bounded, if simply \enquote{$|c(\lambda)|$ is sufficiently large}.
\end{remark}

\subsection{Examples}
We now turn to a more detailed investigation of the conditions for local bifurcation for specific examples of $\gamma$ and $F$.
\subsubsection{No vorticity, no swirl}
As a first example, we consider the case without vorticity and swirl, that is, $\gamma=F=0$. By \eqref{eq:trivial_integral_equation} and \eqref{eq:psi^lambda}, the trivial solutions are given by
\[\psi^\lambda(s)=\lambda.\]
Thus,
\[c(\lambda)=2\lambda,\]
that is, $c(\lambda)\ne 0$ if and only if $\lambda\ne 0$. Moreover, $\tilde\beta=\tilde\beta^{-(k\nu)^2,\lambda}$ solves
\[\left(\partial_s^2+\frac3s\partial_s-(k\nu)^2d^2\right)\tilde\beta=0\text{ on }(0,1],\quad\tilde\beta_s(0)=0,\quad\tilde\beta(0)=1.\]
The general solution to the ODE is given by
\[\tilde\beta(s)=c_1\frac{I_1(k\nu ds)}{s}+c_2\frac{K_1(k\nu ds)}{s},\quad c_1,c_2\in\R,\]
where $I_1$ and $K_1$ are modified Bessel functions of the first and second kind. Since $K_1(x)\to\infty$ as $x\to 0$, we necessarily have $c_2=0$. Determining the remaining constant $c_1$ yields
\[\tilde\beta^{-(k\nu)^2,\lambda}(s)=\frac{2I_1(k\nu ds)}{k\nu ds}\]
and
\[\beta^{-(k\nu)^2,\lambda}(s)=\frac{I_1(k\nu ds)}{sI_1(k\nu d)}.\]
Therefore, using $dI_1/dx=I_0-I_1/x$ (cf. \cite{Amos}),
\[\beta_s^{-(k\nu)^2,\lambda}(1)=\SL\left(\frac{k\nu d\left(I_0(kds)-\frac{I_1(k\nu ds)}{k\nu ds}\right)}{sI_1(k\nu d)}-\frac{I_1(k\nu ds)}{s^2I_1(k\nu d)}\right)=\frac{k\nu dI_0(k\nu d)}{I_1(k\nu d)}-2.\]
Thus, we have
\[\D(-(k\nu)^2,\lambda)=\frac{k\nu dI_0(k\nu d)}{I_1(k\nu d)}+\frac{\sigma}{dc(\lambda)^2}(1-(k\nu)^2d^2).\]
Noticing that necessarily $(k\nu)^2d^2-1>0$ if $\D(-(k\nu)^2,\lambda)=0$, the dispersion relation $\D(-(k\nu)^2,\lambda)=0$ can hence be written as
\begin{align}\label{eq:disprel_no_vortswirl}
	\frac{\sigma}{c(\lambda)^2}=\frac{k\nu d^2I_0(k\nu d)}{((k\nu)^2d^2-1)I_1(k\nu d)}.
\end{align}
This dispersion relation was also obtained in \cite{VB-M-S}. Clearly, in order find solutions of \eqref{eq:disprel_no_vortswirl}, we can first choose arbitrary $\nu>0$, $k\in\N$ with $k\nu>1/d$ and then $\lambda$ such that \eqref{eq:disprel_no_vortswirl} holds. This gives exactly two possible choices $\pm\lambda_0$ for $\lambda$, which correspond to \enquote{mirrored} uniform laminar flows. It is important to notice that, given $c(\lambda)\neq 0$, \eqref{eq:disprel_no_vortswirl} is solved by at most one $k\in\N$; consequently, the kernel of $\LL(\lambda)$ is one-dimensional if this relation is satisfied for some $k\in\N$ and is trivial if it fails to hold for all $k$. Indeed, \eqref{eq:disprel_no_vortswirl} obviously cannot hold for $k\nu d\le 1$; moreover, the function \[g(x)\coloneqq\frac{xI_0(x)}{(x^2-1)I_1(x)},\quad x>1,\]
is strictly monotone on $(1,\infty)$ since
\[g'(x)=\frac{x(x^2-1)(I_1(x)^2-I_0(x)^2)-2I_0(x)I_1(x)}{(x^2-1)^2I_1(x)^2}<0,\quad x>1,\]
as $dI_0/dx=I_1$ and $I_0\ge I_1>0$ on $(0,\infty)$; see \cite{Amos}.

Furthermore, it is therefore clear that the transversality condition $\D_\lambda(-(k\nu)^2,\lambda)\ne 0$ always holds in view of $c_\lambda(\lambda)\ne 0$.

Moreover, it is easy to see that \eqref{ass:realspec_algsimple} is always satisfied since here $q^\lambda=0$ and $h(\lambda)=d^{-2}+\sigma^{-1}d^{-1}c(\lambda)^2>0$.
\subsubsection{Constant $\gamma$, no swirl}
Now let us assume that $\gamma\ne 0$ is a constant and $F=0$. By \eqref{eq:trivial_integral_equation} and \eqref{eq:psi^lambda}, the trivial solutions are given by
\[\psi^\lambda(s)=\lambda-\gamma\int_0^{sd}t^{-3}\int_0^t\tau^3\,d\tau\,dt=\lambda-\frac{\gamma d^2}{8}s^2.\]
Thus,
\begin{align}\label{eq:c_gammaconst}
	c(\lambda)=2\lambda-\frac{\gamma d^2}{2}
\end{align}
that is, $c(\lambda)\ne 0$ if and only if $\lambda\ne\frac{\gamma d^2}{4}$. Noticing that $\beta^{-(k\nu)^2,\lambda}$ is the same as in the previous example without vorticity, we moreover have
\[\D(-(k\nu)^2,\lambda)=\frac{k\nu dI_0(kd)}{I_1(k\nu d)}+\frac{\sigma}{dc(\lambda)^2}(1-(k\nu)^2d^2)+\frac{d^2\gamma}{c(\lambda)}.\]
In order to solve the equation $\D(-(k\nu)^2,\lambda)=0$ for $1/c(\lambda)$, we see that necessarily
\[1-(k\nu)^2d^2=0\]
or
\begin{align}\label{eq:const_vort_cond}
	1-(k\nu)^2d^2\le\frac{d^4\gamma^2I_1(k\nu d)}{4\sigma k\nu I_0(k\nu d)};
\end{align}
obviously, the first case can only occur if $1/(\nu d)\in\N$. We now want to reformulate the second case. Clearly, \eqref{eq:const_vort_cond} holds if $k\nu d\ge 1$. Let us consider $k\nu d<1$ further. The function
\[\chi(x)\coloneqq\frac{I_1(x)}{x(1-x^2)I_0(x)},\quad 0<x<1,\]
is positive and satisfies, using the result
\begin{align}\label{eq:Bessel_ratio_est}
	\frac{x}{1+\sqrt{x^2+1}}\le\frac{I_1(x)}{I_0(x)}\le\frac{x}{\sqrt{x^2+4}}
\end{align}
of \cite{Amos},
\begin{align*}
	\chi_x(x)&=\frac{(1-x^2)\left(x(1-(I_1(x)/I_0(x))^2)-2I_1(x)/I_0(x)\right)+2x^2I_1(x)/I_0(x)}{x^2(1-x^2)^2}\\
	&\ge\frac{(1-x^2)\left(x\left(1-\frac{x^2}{x^2+4}\right)-\frac{2x}{\sqrt{x^2+4}}\right)+\frac{2x^3}{1+\sqrt{x^2+1}}}{x^2(1-x^2)^2}>0,\qquad 0<x<1;
\end{align*}
here, the last inequality follows from the fact that the numerator is positive at $x=1$ and a nonzero root of it, after some algebra, has to satisfy 
\[36x^8+116x^6-64x^4-489x^2-224=0,\]
which can obviously not hold true for $x\in(0,1)$ in view of $36+116<224$. Therefore and because of $I_0(0)=1$, $I_1(0)=0$, and $I_1'(0)=1/2$, the function $\chi\colon (0,1)\to (1/2,\infty)$
is strictly monotonically increasing and onto. Hence, \eqref{eq:const_vort_cond} is always satisfied if $4\sigma d^{-5}\gamma^{-2}\le 1/2$. Otherwise, let $x_1\in(0,1)$ such that $\chi(x_1)=4\sigma d^{-5}\gamma^{-2}$ and $x_0\coloneqq x_1/d$. Thus, we have the equivalence
\[1-(k\nu)^2d^2\left\{\begin{matrix}<\\=\end{matrix}\right\}\frac{d^4\gamma^2I_1(k\nu d)}{4\sigma kI_0(k\nu d)}\Longleftrightarrow k\nu\left\{\begin{matrix}>\\=\end{matrix}\right\}x_0.\]
To conclude, solving $\D(-(k\nu)^2,\lambda)=0$ for $c(\lambda)$ yields
\begin{align}
	c(\lambda)&=-\frac{d^2\gamma I_1(1)}{I_0(1)}&\text{if }k\nu=1/d,\label{eq:disprel_const_vort_1}\\
	c(\lambda)&=\frac{2\sigma((k\nu)^2d^2-1)}{d\left(d^2\gamma\pm\sqrt{d^4\gamma^2+\frac{4\sigma k\nu((k\nu)^2d^2-1)I_0(k\nu d)}{I_1(k\nu d)}}\right)}&\text{if }8\sigma\le d^5\gamma^2,k\nu\ne 1/d,\nonumber\\
	&&\text{or }\text{if }8\sigma>d^5\gamma^2,k\nu\ge x_0,k\nu\ne 1/d,\label{eq:disprel_const_vort_2}
\end{align}
and else, $\D(-(k\nu)^2,\lambda)$ cannot vanish. Next, we search for solutions of \eqref{eq:disprel_const_vort_1} and \eqref{eq:disprel_const_vort_2}. First notice that in both cases it suffices to find appropriate $k$ and $c(\lambda)$ (and not $k$ and $\lambda$) since $\R\ni\lambda\mapsto c(\lambda)\in\R$ is bijective. Both for the first case (for which $1/(\nu d)\in\N$ is necessary) and for the second case, we can easily first choose an appropriate $k$
and then $c(\lambda)$ via \eqref{eq:disprel_const_vort_1} or \eqref{eq:disprel_const_vort_2}. The more interesting question is whether there can be multiple solutions for $k$ for fixed $\lambda$. Clearly, it suffices to focus on the second case. To this end, let us introduce $x=k\nu d$ and write \eqref{eq:disprel_const_vort_2} as $c(\lambda)=b^\pm(x)$ with
\[b^\pm(x)=\frac{2\sigma}{d^3\gamma}\frac{x^2-1}{1\pm\sqrt{1+\xi\frac{x(x^2-1)I_0(x)}{I_1(x)}}},\]
where
\[\xi\coloneqq\frac{4\sigma}{d^5\gamma^2}.\]
Here, $b^-(1)$ and possibly $b^\pm(0)$ are to be interpreted as the limit of the above expression as $x$ tends to $1$ or $0$; the limit for $x\to 1$ exists since $x=1$ is a simple root of both nominator and denominator, and the limit for $x\to 0$ also as $I_1(0)=0$ and $I_1'(0)=1/2$. Having clarified this, we see that $b^\pm$ is smooth on $(x_1,\infty)$ and continuous on $[x_1,\infty)$ if $\xi>1/2$, smooth on $(0,\infty)$ and continuous on $[0,\infty)$ if $\xi=1/2$, and smooth on $[0,\infty)$ if $\xi<1/2$. Now notice that it obviously suffices to consider $\gamma>0$ in the following without loss of generality.

We have
\begin{align}\label{eq:disp_rel_bpm}
	0=(b^\pm(x))^2\D(-(k\nu)^2,\lambda)=b^\pm(x)^2f(x)+\frac{\sigma}{d}(1-x^2)+d^2\gamma b^\pm(x),
\end{align}
where
\[f(x)\coloneqq\frac{xI_0(x)}{I_1(x)}.\]
Differentiating \eqref{eq:disp_rel_bpm} with respect to $x$ yields
\begin{align}
	b^\pm_x(2b^\pm f+d^2\gamma)&=-(b^\pm)^2f_x+\frac{2\sigma x}{d},\label{eq:b_x}\\
	b^\pm_{xx}(2b^\pm f+d^2\gamma)&=-4b^\pm b^\pm_xf_x-2(b^\pm_x)^2f-(b^\pm)^2f_{xx}+\frac{2\sigma}{d}.\label{eq:b_xx}
\end{align}
Thus, if $b^\pm_x=0$ at some $x>0$, then
\begin{align*}
	b^\pm_{xx}(2b^\pm f+d^2\gamma)=\frac{2\sigma}{d}-(b^\pm)^2f_{xx}=\frac{2\sigma}{d}\frac{f_x-xf_{xx}}{f_x}.
\end{align*}
Here, we notice that $f_x>0$ for $x>0$ because of
\[f_x(x)=x\left(1-\frac{I_0(x)/I_1(x)}{I_1(x)/I_2(x)}\right)\ge x\left(1-\frac{1+\sqrt{x^2+1}}{1+\sqrt{x^2+9}}\right)>0\]
due to \cite{Amos}. Moreover,
\begin{align}\label{eq:b_bracket_sign}
	2b^\pm(x)f(x)+d^2\gamma=\pm d^2\gamma\sqrt{1+\xi\frac{x(x^2-1)I_0(x)}{I_1(x)}}\gtrless 0.
\end{align}
Furthermore, we have
\begin{align}\label{eq:fx-xfxx}
	f_x(x)-xf_{xx}(x)=-\frac{2x^2I_0(x)^3}{I_1(x)^3}+\frac{4xI_0(x)^2}{I_1(x)^2}+\frac{2x^2I_0(x)}{I_1(x)}-2x>0.
\end{align}
Instead of presenting a lengthy, not very instructive proof of this inequality we provide a plot of the left-hand side (multiplied by a suitable positive function) in Figure \ref{fig:besselplot} in order to convince the reader of the validity of \eqref{eq:fx-xfxx}.
\begin{figure}[h!]
	\centering
	\includegraphics{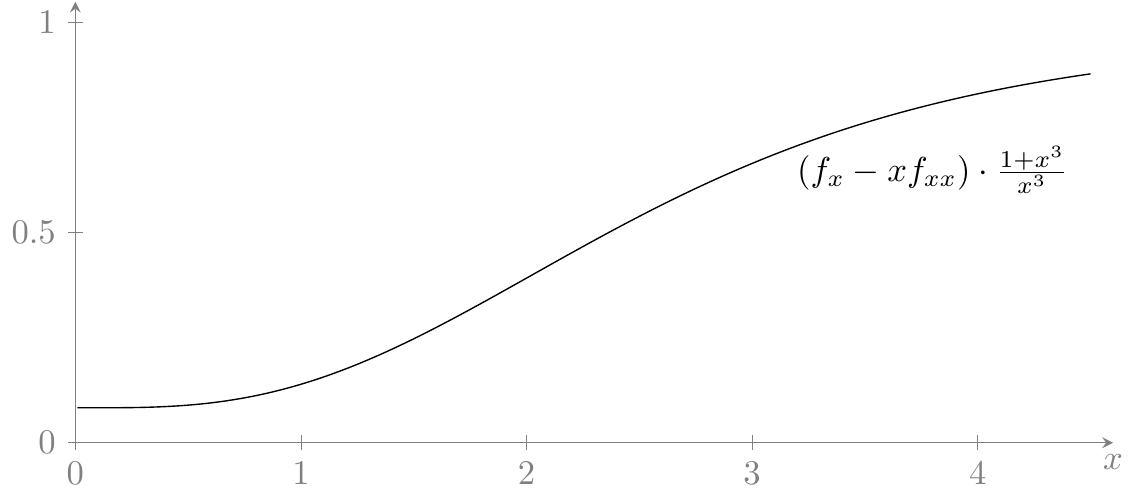}
	\caption{Demonstration of the validity of \eqref{eq:fx-xfxx}.}
	\label{fig:besselplot}
\end{figure}

Thus, putting everything together, $b^\pm_{xx}\gtrless 0$ provided $b^\pm_x=0$. In particular, $b^\pm$ can have at most one critical point on $(0,\infty)$, which, if it exists, has to be a local minimum (maximum). Since moreover $b^\pm$ tends to $\pm\infty$ as $x\to\infty$ by \eqref{eq:Bessel_ratio_est}, we conclude that the monotonicity properties of $b^\pm$ can be characterised by its behaviour near $0$ if $\xi\le1/2$ or near $x_1$ if $\xi>1/2$.

The easy case is $\xi>1/2$. Since
\[b^+(x_1)=b^-(x_1)=\frac{2\sigma(x_1^2-1)}{d^3\gamma},\quad\lim_{\substack{x\to x_1\\x>x_1}}b^\pm_x(x)=\pm\infty\]
due to $\chi_x(x_1)\ne 0$, we conclude that $b^\pm$ is strictly monotonically increasing (decreasing) on $[x_1,\infty)$ and $b^+((x_1,\infty))\cap b^-((x_1,\infty))=\emptyset$.

If $\xi=1/2$, we can argue similarly. Still we have $b^+(0)=b^-(0)=-2\sigma d^{-3}\gamma^{-1}$, but $b^\pm_x$ remains bounded as $x\to 0$. Indeed, from the Taylor expansion
\[1+\frac{x(x^2-1)I_0(x)}{2I_1(x)}=\frac78x^2+\mathcal O(x^4),\quad x\to 0,\]
we infer that
\[\lim_{\substack{x\to 0\\x>0}}b^\pm_x(x)=\pm\frac{2\sigma}{d^3\gamma}\cdot\frac{7/4}{2\sqrt{7/8}}\gtrless0.\]
Therefore, the same conclusions hold as before, namely, $b^\pm$ is strictly monotonically increasing (decreasing) on $[0,\infty)$ and $b^+((0,\infty))\cap b^-((0,\infty))=\emptyset$.

Let us now turn to the case $\xi<1/2$ and take a look at $x=0$. By \eqref{eq:b_xx}, \eqref{eq:b_bracket_sign}, and $b^\pm_x(0)=0$ because of evenness, we see that $b^\pm_{xx}(0)$ has the same sign as $\pm(2\sigma/d-(b^\pm(0))^2f_{xx}(0))$, or vanishes if and only if $2\sigma/d=(b^\pm(0))^2f_{xx}(0)$. Now
\[\frac{2\sigma}{d}-(b^\pm(0))^2f_{xx}(0)=\frac{2\sigma}{d}-\left(\frac{2\sigma}{d^3\gamma}\cdot\frac{-1}{1\pm\sqrt{1-2\xi}}\right)^2\cdot\frac12=\frac{\sigma}{4d\xi}(9\xi-1\pm\sqrt{1-2\xi})\eqqcolon\frac{\sigma}{4d\xi}g^\pm(\xi).\]
First, because of
\[g^+(0)=0,\quad g^+(1/2)=7/2>0,\quad g^+_{\xi\xi}(\xi)=-\frac{1}{(1-2\xi)^{3/2}}<0,\]
$g^+$ is positive on $(0,1/2)$. Therefore, for any $\xi<1/2$, $b^+$ is strictly monotonically increasing on $[0,\infty)$. Second, we have
\[g^-(0)=-2<0,\quad g^-(1/2)=7/2>0,\quad g^-_\xi(\xi)=9+\frac{1}{\sqrt{1-2\xi}}>0,\quad g^-(16/81)=0,\]
and thus
\[g^-(\xi)\begin{cases}<0,&0\le\xi<16/81,\\=0,&\xi=16/81,\\>0,&16/81<\xi<1/2.\end{cases}\]
Hence, $b^-$ is strictly monotonically decreasing on $[0,\infty)$ if $\xi>16/81$ and has exactly one local extremum (which is in fact a global maximum) if $\xi<16/81$. We moreover want to prove that $\max b^-<\min b^+$, that is, $\max b^-<b^+(0)$. To this end, first notice that $b^-<0$ on $[0,\infty)$ since both the nominator and denominator in the definition of $b^-$ have a simple root at $x=1$ and thus $b^-$ cannot have a zero. By \eqref{eq:b_x} we therefore have
\[\max b^-\le -\inf_{x>0}\sqrt{\frac{2\sigma x}{df_x(x)}}=-\sqrt{\frac{2\sigma}{d}}\inf_{x>0}\frac{1}{\sqrt{1-I_0(x)I_2(x)/I_1(x)^2}}\le-\sqrt{\frac{2\sigma}{d}}.\]
Hence,
\[b^+(0)=-\frac{2\sigma}{d^3\gamma(1+\sqrt{1-2\xi})}=-\sqrt{\frac{\sigma}{d}}\cdot\frac{\sqrt{\xi}}{1+\sqrt{1-2\xi}}>-\frac14\sqrt{\frac{\sigma}{d}}>-\sqrt{\frac{2\sigma}{d}}\ge\max b^-,\]
since $\xi\mapsto\sqrt\xi/(1+\sqrt{1-2\xi})$ is strictly monotonically increasing on $[0,16/81]$.

Let us now consider $\xi=16/81$. Differentiating \eqref{eq:b_xx} twice more, evaluating at $0$, and using $b^-_x(0)=b^-_{xx}(0)=b^-_{xxx}(0)=0$ yields
\[b^-_{xxxx}(0)(2b^-(0)f(0)+d^2\gamma)=-b^-(0)^2f_{xxxx}(0)=\frac14b^-(0)^2>0.\]
In particular, $b^-_{xxxx}(0)<0$; hence, $b^-$ is strictly monotonically decreasing.

To summarise, for fixed $\lambda$ we have therefore proved the following, provided $1/(\nu d)\notin\N$; below in Figure \ref{fig:plots_b} the respective cases are visualised:
\begin{itemize}
	\item If $\xi\ge 16/81$:
	\begin{itemize}
		\item The dispersion relation $\D(-(k\nu)^2,\lambda)=0$ can have at most one root $k\in\N$.
		\item If $16/81\le\xi<1/2$ and
		\[-\frac{2\sigma}{d^3\gamma(1-\sqrt{1-2\xi})}<c(\lambda)<-\frac{2\sigma}{d^3\gamma(1+\sqrt{1-2\xi})},\]
		the dispersion relation has no root.
	\end{itemize}
	\item If $\xi<16/81$:
	\begin{itemize}
		\item If
		\[c(\lambda)>-\frac{2\sigma}{d^3\gamma(1+\sqrt{1-2\xi})}\quad\text{or}\quad c(\lambda)=\max b^-\quad\text{or}\quad c(\lambda)\le-\frac{2\sigma}{d^3\gamma(1-\sqrt{1-2\xi})},\]
		the dispersion relation has at most one root.
		\item If
		\[\max b^-<c(\lambda)\le-\frac{2\sigma}{d^3\gamma(1+\sqrt{1-2\xi})},\]
		the dispersion relation has no root.
		\item If
		\[-\frac{2\sigma}{d^3\gamma(1-\sqrt{1-2\xi})}<c(\lambda)<\max b^-,\]
		the dispersion relation has at most two roots.
	\end{itemize}
\end{itemize}
\begin{figure}[h!]
	\centering
	\subfigure[$\xi\ge1/2$.]{\includegraphics[width=0.33\columnwidth]{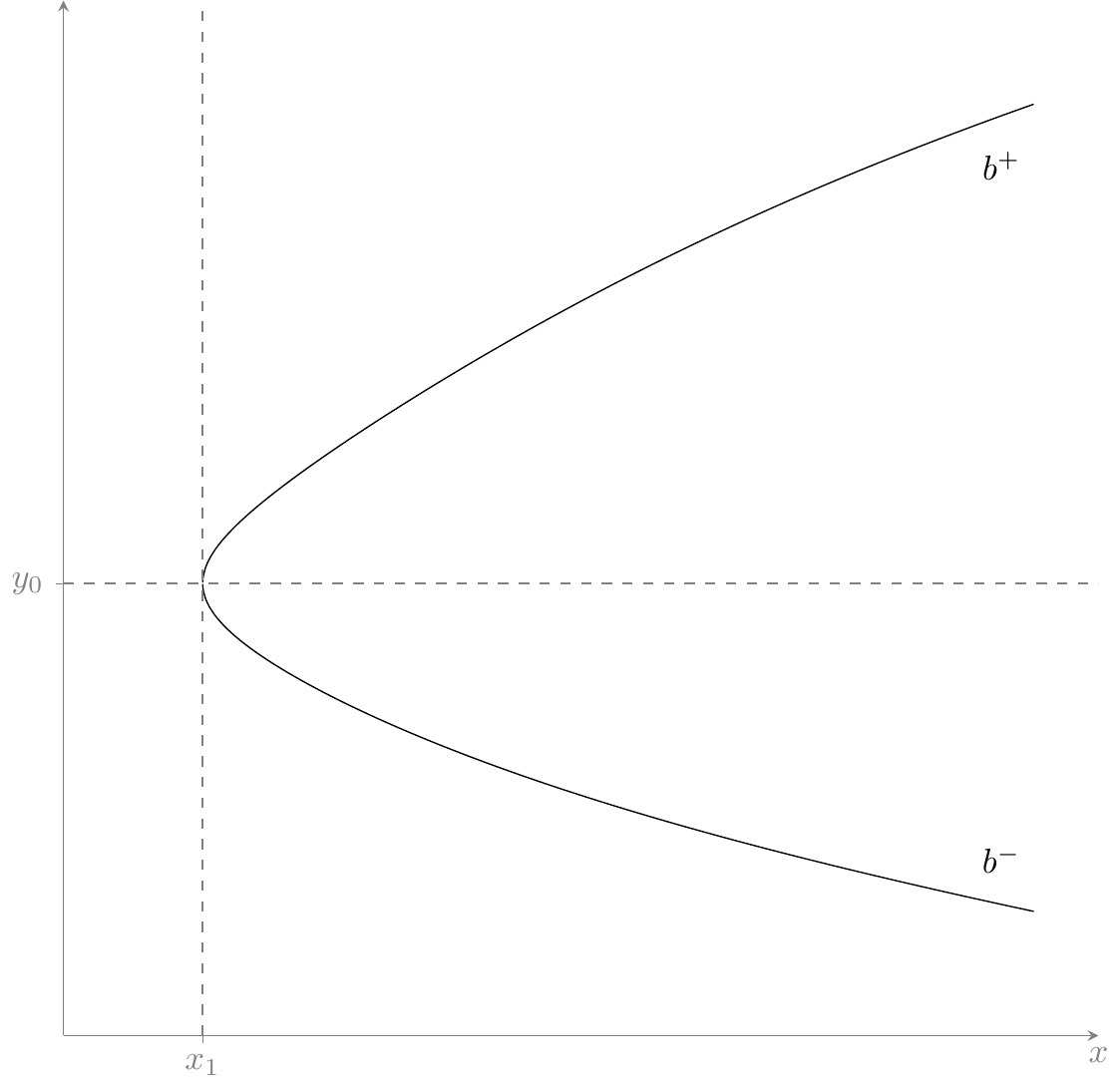}}\subfigure[$16/81\le\xi<1/2$.]{\includegraphics[width=0.33\columnwidth]{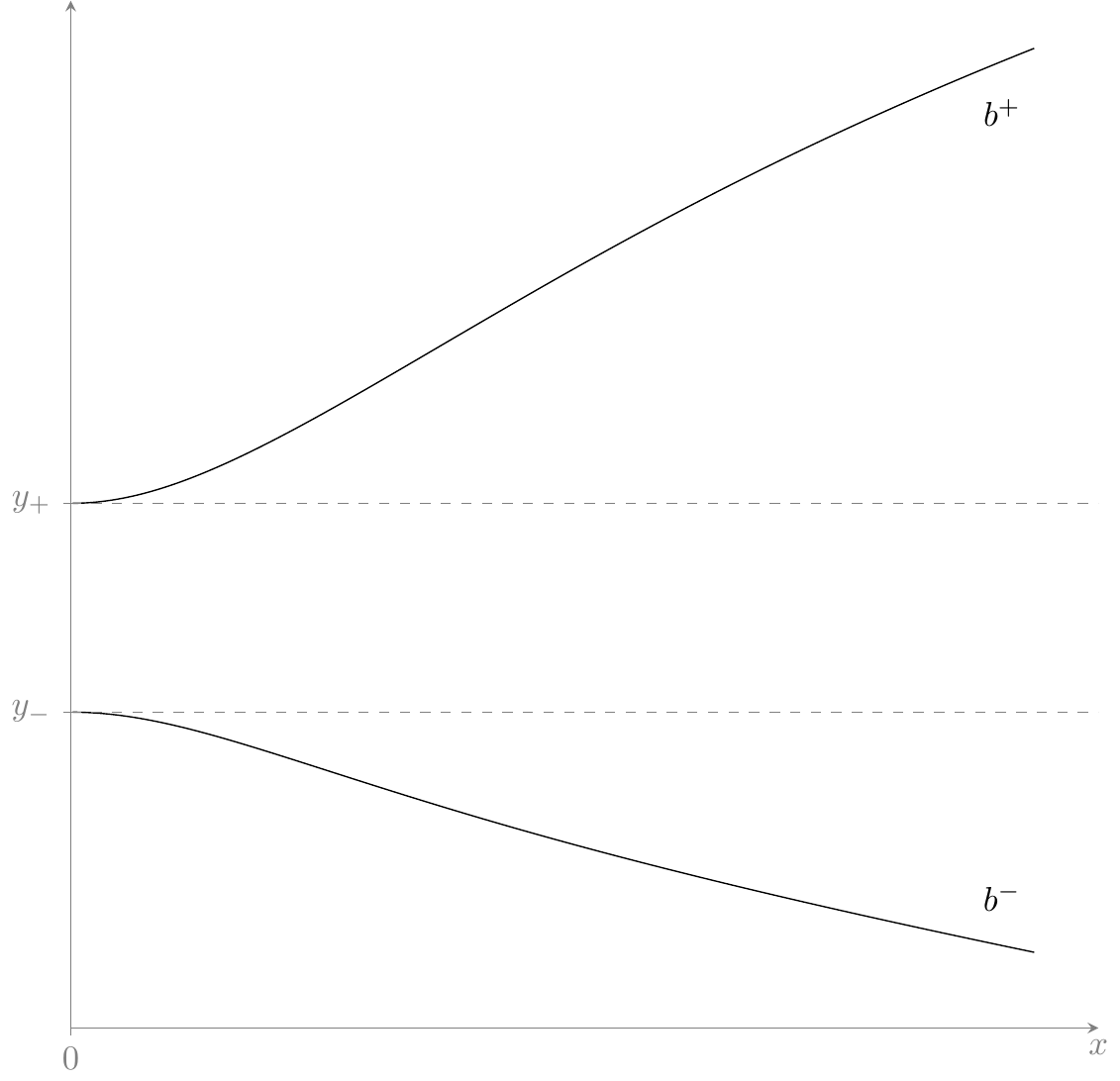}}\subfigure[$0<\xi<16/81$.]{\includegraphics[width=0.33\columnwidth]{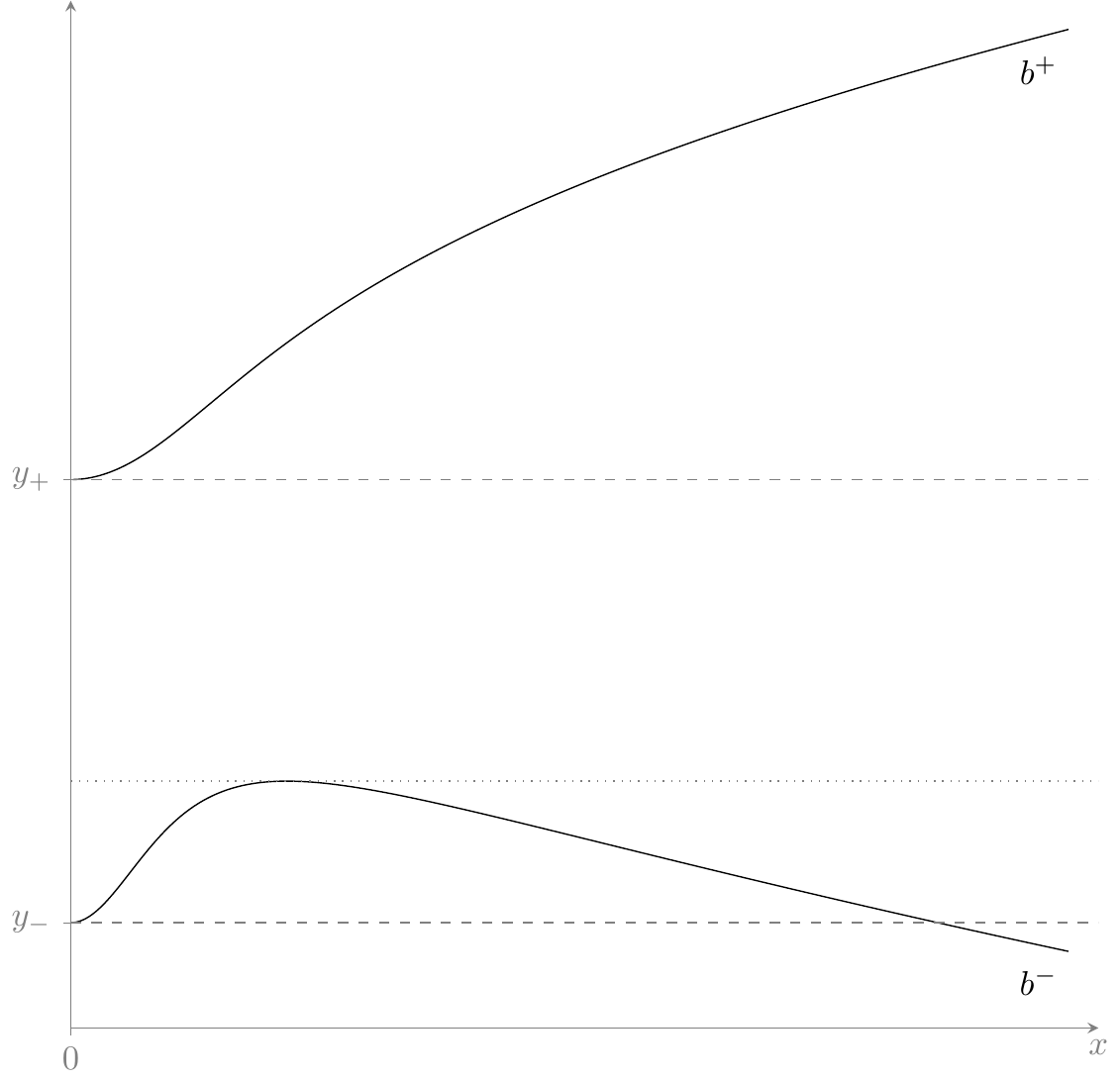}}
	\caption{Qualitative behaviour of $b^\pm$ for different $\xi$ (in the case $\gamma>0$). Here, $y_0\coloneqq2\sigma d^{-3}\gamma^{-1}(x_1^2-1)$ and $y_\pm\coloneqq-2\sigma d^{-3}\gamma^{-1}/(1\pm\sqrt{1-2\xi})$.}
	\label{fig:plots_b}
\end{figure}
If $1/(\nu d)\in\N$ and
\[c(\lambda)=-\frac{d^2\gamma I_1(1)}{I_0(1)},\]
there is the additional root $k=1/(\nu d)$.

If, however, $\gamma<0$, these statements remain true after reversing all inequalities in the conditions for $c(\lambda)$ and changing $\max b^-$ to $\min b^-$.

Next, let us turn to the transversality condition, fix $\lambda$, and assume that $\D(-(k\nu)^2,\lambda)=0$ has exactly one solution $k\in\N$. Since $c_\lambda(\lambda)\ne 0$, it holds that $\D_\lambda(-(k\nu)^2,\lambda)\ne 0$ if and only if $\xi\le1/2$ or $k\nu d>x_1$ otherwise.

Finally, we have a look at \eqref{ass:realspec_algsimple}. Here, $q^\lambda=0$ and $h(\lambda)=d^{-2}+2\sigma^{-1}d^{-1}\lambda(4\lambda-\gamma d^2)$ by \eqref{eq:c_gammaconst}. Therefore, $h(\lambda)>0$ for all $\lambda\in\R$ if $\gamma^2<8\sigma d^{-5}\Leftrightarrow\xi>\frac12$, and in the case $\xi\le\frac12$, $h(\lambda)>0$ if and only if $\lambda\notin[\lambda_-,\lambda_+]$ where
\[\lambda_\pm\coloneqq\frac{|\gamma|d^2(\sgn\gamma\pm\sqrt{1-2\xi})}{8}.\]

\section{Global bifurcation}\label{sec:GlobalBifurcation}
The theory for local bifurcation having set up, we now turn to global bifurcation, which is of course the main motivation of our formulation \enquote{identity plus compact}. To this end, we first state the global bifurcation theorem by Rabinowitz.
\begin{theorem}\label{thm:Rabinowitz}
	Let $X$ be a Banach space, $U\subset\R\times X$ open, and $\F\in C(U;X)$. Assume that $\F$ admits the form $\F(\lambda,x)=x+f(\lambda,x)$ with $f$ compact, and that $\F_x(\cdot,0)\in C(\R;L(X,X))$. Moreover, suppose that $\F(\lambda_0,0)=0$ and that $\F_x(\lambda,0)$ has an odd crossing number at $\lambda=\lambda_0$. Let $S$ denote the closure of the set of nontrivial solutions of $\F(\lambda,x)=0$ in $\R\times X$ and $\mathcal C$ denote the connected component of $S$ to which $(\lambda_0,0)$ belongs. Then one of the following alternatives occurs:
	\begin{enumerate}[label=(\roman*)]
		\item $\mathcal C$ is unbounded;
		\item $\mathcal C$ contains a point $(\lambda_1,0)$ with $\lambda_1\neq\lambda_0$;
		\item $\mathcal C$ contains a point on the boundary of $U$.
	\end{enumerate}
\end{theorem}
The proof of this theorem in the case $U=X$ can be found in \cite[Theorem II.3.3]{Kielhoefer} and is practically identical to the proof for general $U$.

Now we can prove the following result.

\begin{theorem}\label{thm:GlobalBifurcation}
	Assume \eqref{ass:SL-spectrum} and that there exists $\lambda_0\neq 0$ such that the dispersion relation
	\[\D(-(k\nu)^2,\lambda_0)=0,\]
	with $\D$ given by \eqref{eq:d(k,lambda)}, has exactly one solution $k_0\in\N$ and assume that the transversality condition
	\[\D_\lambda(-(k_0\nu)^2,\lambda_0)\neq 0\]
	holds. Let $S$ denote the closure of the set of nontrivial solutions of $\F(\lambda,\eta,\phi)=0$ in $\R\times X$ and $\mathcal C$ denote the connected component of $S$ to which $(\lambda_0,0,0)$ belongs. Then one of the following alternatives occurs:
	\begin{enumerate}[label=(\roman*)]
		\item $\mathcal C$ is unbounded in the sense that there exists a sequence $(\lambda_n,\eta_n,\phi_n)\in\mathcal C$ such that
		\begin{enumerate}[label=(\alph*)]
			\item $|\lambda_n|\to\infty$, or
			\item $\|\eta_n\|_{C^{2,\alpha}([0,L])}\to\infty$, or
			\item $\|r^{2/p}\gamma(\Psi_n)+r^{2/p-2}F(\Psi_n)F'(\Psi_n)\|_{L^p(Z_{\eta_n})}\to\infty$ with $p\coloneqq\frac{5}{2-\alpha}$, where $Z_{\eta_n}$ denotes a $L$-periodic instance of the axially symmetric fluid domain in $\R^3$ corresponding to $\eta_n$ and $\Psi_n=r^2\left(\left(\phi_n+\frac{d^2}{(d+\eta_n)^2}\psi^{\lambda_n}\right)\circ H[\eta_n]^{-1}\right)$ is the corresponding original Stokes stream function,
		\end{enumerate}
		as $n\to\infty$;
		\item $\mathcal C$ contains a point $(\lambda_1,0,0)$ with $\lambda_1\neq\lambda_0$;
		\item $\mathcal C$ contains a sequence $(\lambda_n,\eta_n,\phi_n)$ such that $\eta_n$ converges to some $\eta$ in $C^{2,\beta}_{0,\per,\e}(\R)$ for any $\beta\in(0,\alpha)$ and such that there exists $z\in[0,L]$ with
		\[\eta(z)=-d,\]
		that is, intersection of the surface profile with the cylinder axis occurs.
	\end{enumerate}
\end{theorem}
\begin{proof}
	As was already observed in Lemma \ref{lma:M_prop}, our nonlinear operator $\F$ is of class $C^2$ and admits the form \enquote{identity plus compact} on each $\R\times\U_\varepsilon$, $\varepsilon>0$. Moreover, it is well-known that $F_{(\eta,\phi)}(\lambda,\eta,\phi)$ has an odd crossing number at $(\lambda_0,0,0)$ provided $F_{(\eta,\phi)}(\lambda_0,0,0)$ is a Fredholm operator with index zero and one-dimensional kernel, and the transversality condition holds. These properties, in turn, are consequences of the hypotheses of the theorem in view of Lemmas \ref{lma:kernel} and \ref{lma:transversality_condition} since $F_{(\eta,\phi)}(\lambda_0,0,0)$ coincides with $\LL(\lambda_0)$ up to an isomorphism. For each $\varepsilon>0$, we can thus apply Theorem \ref{thm:Rabinowitz} with $U$ chosen to be the interior of $\R\times\U_\varepsilon$. Thus, on each $\R\times\U_\varepsilon$, $\mathcal C$ coincides with its counterpart obtained from Theorem \ref{thm:Rabinowitz}. Since $\varepsilon>0$ is arbitrary and $\R\times\U=\bigcup_{\varepsilon>0}(\R\times\U_\varepsilon)$, it is evident that necessarily 
	\begin{align}\label{eq:alternative_inf}
		\inf_{(\lambda,\eta,\phi)\in\Cc}\min_\R(\eta+d)=0
	\end{align}
	whenever $\mathcal C$ is bounded in $\R\times X$ and (ii) fails to hold.
	
	Let us investigate alternative (i) further. In order to show that it can be as stated above, we show that, in view of alternative (i) of Theorem \ref{thm:Rabinowitz}, $\mathcal C$ is bounded in $\R\times X$ if (i)(a)--(c) and \eqref{eq:alternative_inf} fail to hold. Indeed, along $\mathcal C$ we have $\phi=\A(\lambda,\eta,\phi)$ and, since \eqref{eq:alternative_inf} does not hold, $\eta+d\ge\varepsilon$ uniformly for some $\varepsilon>0$. Thus,
	\begin{align*}
	&\|\phi\|_{C^{0,\alpha}_\per(\overline{\Omega_0})}+\|\I\phi\|_{H^1_\per(\Omega_0^\I)}\le\|\I\phi\|_{C^{0,\alpha}_\per(\overline{\Omega_0^\I})}+\|\I\phi\|_{H^1_\per(\Omega_0^\I)}\le C\|\I\phi\|_{W^{2,p}(\tilde\Omega_0^\I)}\\
	&\le C\Bigg(\|\eta\|_{C^{2,\alpha}([0,L])},\varepsilon^{-1},\Bigg\|\gamma\left((d+\eta)^2|y|^2\left(\I\phi+\frac{d^2}{(d+\eta)^2}\I\psi^\lambda\right)\right)\\
	&\phantom{\le\;c\Bigg(\|\eta\|_{C^{2,\alpha}([0,L])},}+\frac{1}{(d+\eta)^2|y|^2}(FF')\left((d+\eta)^2|y|^2\left(\I\phi+\frac{d^2}{(d+\eta)^2}\I\psi^\lambda\right)\right)+L^\eta\frac{d^2\I\psi^\lambda}{(d+\eta)^2}\Bigg\|_{L^p(\tilde\Omega_0^\I)}\Bigg)\\
	&\le C\Bigg(\|\eta\|_{C^{2,\alpha}([0,L])},\varepsilon^{-1},|\lambda|,\Bigg\|s^{3/p}\Bigg[\gamma\left((d+\eta)^2s^2\left(\phi+\frac{d^2}{(d+\eta)^2}\psi^\lambda\right)\right)\\
	&\omit\hfill$\displaystyle+\frac{1}{(d+\eta)^2s^2}(FF')\left((d+\eta)^2s^2\left(\phi+\frac{d^2}{(d+\eta)^2}\psi^\lambda\right)\right)\Bigg]\Bigg\|_{L^p(\tilde\Omega_0)}\Bigg)$\\
	&\le C\Bigg(\|\eta\|_{C^{2,\alpha}([0,L])},\varepsilon^{-1},|\lambda|,\Bigg\|\Bigg(s^{3/p}\Bigg[\gamma\left((d+\eta)^2s^2\left(\phi+\frac{d^2}{(d+\eta)^2}\psi^\lambda\right)\right)\\
	&\omit\hfill$\displaystyle+\frac{1}{(d+\eta)^2s^2}(FF')\left((d+\eta)^2s^2\left(\phi+\frac{d^2}{(d+\eta)^2}\psi^\lambda\right)\right)\Bigg]\Bigg)\circ H[\eta]^{-1}\Bigg\|_{L^p(\tilde\Omega_\eta)}\Bigg)$\\
	&\le C\left(\|\eta\|_{C^{2,\alpha}([0,L])},\varepsilon^{-1},|\lambda|,\left\|r^{3/p}\left[\gamma(\Psi)+\frac{1}{r^2}(FF')(\Psi)\right]\right\|_{L^p(\tilde\Omega_\eta)}\right)\\
	&\le C\left(\|\eta\|_{C^{2,\alpha}([0,L])},\varepsilon^{-1},|\lambda|,\left\|r^{2/p}\gamma(\Psi)+r^{2/p-2}(FF')(\Psi)\right\|_{L^p(Z_\eta)}\right)
	\end{align*}
	after using Sobolev's embedding, the Calderón--Zygmund inequality (see \cite[Chapter 9]{GilbargTrudinger}; notice that on the right-hand side the term $\|\I\phi\|_{L^p(\tilde\Omega_0^\I)}$ can be left out because of unique solvability of the Dirichlet problem associated to $L^\eta$), and changes of variables via $H[\eta]$ and via cylindrical coordinates in $\R^5$ and $\R^3$, and where $\tilde\Omega_\eta$ denotes a periodic instance of $\Omega_{\eta}=H[\eta](\Omega_0)$ and $\Psi$, $Z_\eta$ are analogously defined as in the statement of (c); here, the constant $C>0$ can change in each step.
	
	Finally, we turn to alternative (iii). If \eqref{eq:alternative_inf} holds, but not (i)(b), then clearly we find a sequence as described in (iii) due to the compact embedding of $C_{0,\per,\e}^{2,\alpha}(\R)$ in $C_{0,\per,\e}^{2,\beta}(\R)$.
\end{proof}
\begin{remark}
	Alternative (i)(c) says that the angular component of the vorticity, in general given by $\omega^\vartheta=\vec{\omega}\cdot\vec{e}_\vartheta=-r\gamma(\Psi)-(FF')(\Psi)/r$, satisfies $\|r^{2/p-1}\omega^\vartheta_n\|_{L^p(Z_{\eta_n})}\to\infty$ as $n\to\infty$.
\end{remark}
We also have the following.
\begin{proposition}
	In Theorem \ref{thm:GlobalBifurcation} the alternative (i)(b) can be replaced by
	\begin{enumerate}[leftmargin=40pt]
		\item[(i)(b')]
		\begin{enumerate}[label=(\greek*)]
			\item $\|\eta_n\|_{C^{1,\alpha}([0,L])}\to\infty$, or
			\item $\||\vec{u}_n|^2\|_{C^{0,\alpha}(S_n)}\to\infty$ (the square of the velocity [the kinetic energy density] is unbounded in $C^{0,\alpha}$ at the free surface $S_n$), or
			\item $|Q(\lambda_n,\eta_n,\phi_n)|\to\infty$ (the Bernoulli constant is unbounded).
		\end{enumerate}
	\end{enumerate}
\end{proposition}
\begin{proof}
	This follows easily from the Bernoulli equation
	\[Q(\lambda,\eta,\phi)=\frac12|\vec{u}|^2-\sigma\left(\frac{\eta_{zz}}{(1+\eta_z^2)^{3/2}}-\frac{1}{(d+\eta)\sqrt{1+\eta_z^2}}\right)\]
	at the free surface.
\end{proof}

\small{\textbf{Acknowledgements.} 
This project has received funding from the European Research Council (ERC) under the European Union’s Horizon 2020 research and innovation programme (grant agreement no 678698).

A.E., supported in 2020 by the Kristine Bonnevie scholarship 2020 of the Faculty of Mathematics and Natural Sciences, University of Oslo, during his research stay at Lund University, wishes to thank Erik Wahl\'en and the Centre of Mathematical Sciences, Lund University for hosting him. A.E. was partially supported by the DFG under Germany's Excellence Strategy – MATH$^+$: The Berlin Mathematics Research Center (EXC-2046/1 – project ID: 390685689) via the project AA1-12$^*$}
 
 \bibliographystyle{siam}
\bibliography{bib_axisym}

\end{document}